\documentclass[12pt]{article}
\usepackage{amssymb}
\usepackage{amsthm}
\usepackage{amsmath}
\usepackage{graphicx}
 \newtheorem{theorem}{Theorem}[section]
 
 \newtheorem{lemma}[theorem]{Lemma}
 \newtheorem{proposition}[theorem]{Proposition}
 \theoremstyle{definition}
 
 \newtheorem{remark}[theorem]{Remark}
 \numberwithin{equation}{section}

\begin{document}

\title{Global Solutions for Incompressible Viscoelastic Fluids}
\author{Zhen Lei\footnote{School of Mathematical Sciences, Fudan
  University, Shanghai 200433, China; School of Mathematics and Statistics,
  Northeast Normal University, Changchun 130024, P. R. China. {\it
  Email: leizhn@yahoo.com}} \and Chun Liu\footnote{Department of Mathematics,
Pennsylvania State University, State college, PA 16802, USA.
Email:{liu@math.psu.edu}} \and Yi Zhou\footnote{School of
Mathematical Sciences, Fudan University, Shanghai 200433, China.
{\it Email: yizhou@fudan.ac.cn}}}
\date{}
\maketitle

\begin{abstract}
We prove the existence of both local and global smooth solutions
to the Cauchy problem in the whole space and the periodic problem
in the n-dimensional torus for the incompressible viscoelastic
system of Oldroyd-B type in the case of near equilibrium initial
data. The results hold in both two and three dimensional spaces.
The results and methods presented in this paper are also valid for
a wide range of elastic complex fluids, such as
magnetohydrodynamics, liquid crystals and mixture problems.
\end{abstract}

\section{Introduction}\label{intro}

Many of the rheological and hydrodynamical properties of complex
fluids can be attributed to the competition between the kinetic
energies and the internal elastic energies, through the special
transport properties of their respective internal elastic
variables. Moreover, any distortion of microstructures, patterns
or configurations in the dynamical flow will involve the
deformation tensor $F$. In contrast to the classical simple
fluids, where the internal energies can be determined by solely
the determinant of the deformation tensor $F$, the internal
energies of the complex fluids carry all the information of this
tensor \cite{Larson,Dafermos}.

In this paper we consider the following system describing
incompressible viscoelastic fluids. The existence results  we
obtain in this paper, together with the methods,  are valid in
many related systems,  such as those for general polymeric
materials \cite{Bird,Larson}, magnetohydrodynamics (MHD)
\cite{Davison}, liquid crystals \cite{deGennes,Lin2}, and the free
interface motion in mixture problems \cite{Liu2}. The entire
coupled hydrodynamical system we consider here contains a linear
momentum equation (force balance law), the incompressibility and a
microscopic equation specifying the special transport of the
elastic variable $F$:
\begin{equation}\label{a1}
\left\{
  \begin{array} {l@{\quad \ \quad}l}
    \nabla \cdot v = 0,\\
    v_t + v \cdot \nabla v + \nabla p
      = \mu\Delta v + \nabla\cdot\big[\frac{\partial
      W(F)}{\partial F}F^T\big],\\
    F_t + v \cdot \nabla F = \nabla vF.
  \end{array}
\right.
\end{equation}
Here $v(t, x)$ represents the velocity field of materials, $p(t,
x)$ the pressure, $\mu$ $(> 0)$ the viscosity, $F(t, x)$ the
deformation tensor and $W(F)$ the elastic energy functional. The
third equation is simply the consequence of the chain law. It can
also be regarded as the consistence condition of the flow
trajectories obtained from the velocity field $v$ and  those from
the deformation tensor $F$ \cite{Gurtin,Lin3,Liu,Dafermos}.
Moreover, in the right hand side of the momentum equation,
$\frac{\partial W(F)}{\partial F}$ is the Piola-Kirchhoff stress
tensor and $\frac{\partial W(F)}{\partial F}F^T$ is the
Cauchy-Green tensor, both in the incompressible case. The latter
is the change variable (from Lagrangian coordinate to Eulerian
coordinate) form of the former one.

Throughout this paper we will adopt the notations of
$$(\nabla v)_{ij} =
\frac{\partial v_{i}}{\partial x_{j}},\ \ (\nabla vF)_{ij} =
(\nabla v)_{ik}F_{kj},\ \ (\nabla\cdot F)_i = \partial_jF_{ij},$$
and summation over repeated indices will always be well
understood.

The above system is equivalent to the usual Oldroyd-B model for
viscoelastic fluids in infinite  Weissenberg number cases
\cite{Larson}. On the other hand, without the viscosity term, it
represents exactly the incompressible elasticity in Eulerian
coordinate. We want to refer to
\cite{Lei2,Lei3,Lin1,Lions,Liu,Larson,Bird,Dafermos} and their
references for the detailed derivation and physical background of
the above system.

Due to the elasticity nature of our system (also being regarded as
a first step in understanding the dynamical properties of such
systems), the study of the near equilibrium dynamics of the system
is both relevant and very important. For this purpose, we will
impose the following initial conditions on system (\ref{a1}):
\begin{equation}\label{a2}
F(0, x) = I + E_{0}(x),\ \ \ \ v(0, x) = v_{0}(x),\ \ \ \ x \in
\Omega.
\end{equation}
where $\Omega$ is the physical domain under consideration. We
further assume that $E_{0}(x)$ and $v_{0}(x)$ satisfy the
following  constraints:
\begin{equation}\label{a3}
\left\{
  \begin{array} {l@{\quad \ \quad}l}
   \nabla\cdot v_0 = 0,\\
   \det(I + E_0) = 1,\\
   \nabla\cdot E_0^T = 0,\\
   \nabla_mE_{0ij} - \nabla_jE_{0im} = E_{0lj}\nabla_lE_{0im} -
E_{0lm}\nabla_lE_{0ij}.
  \end{array}
\right.
\end{equation}
The first three are just the consequences of the incompressibility
condition \cite{Liu,Lin3} and the last one can be understood as
the consistency condition for changing of variables between the
Lagrangian and Eulerian coordinates (see Lemma \ref{lem23} and
Remark \ref{rem22}).

When $\Omega$ is a bounded domain with smooth boundary, we will
choose the following Dirichlet boundary conditions:
\begin{equation}\label{a4}
v(t, x) = 0,\ \ \ \ E(t, x) = 0,\ \ \ \ (t, x) \in [0,T) \times
\partial\Omega.
\end{equation}

Global existence of  classical solutions for system (\ref{a1})
with small initial data $E_0$ and $v_0$, for the Cauchy problem in
the whole space and the periodic problem in the n-dimensional
torus $\Omega = T^n$, will be proved in this paper. Our methods in
this paper are independent of the space dimensions. We point out
that the initial-boundary value problem (\ref{a1}), with
(\ref{a2}) and (\ref{a4})  can also be treated at a more lengthy
procedure, with few more technical difficulties than the ones
presented in this paper.

There have been a long history of studies in understanding
different phenomena for non-Newtonian fluids, such as those of
Erichsen-Rivlin models \cite{Schowalter,Slemrod}, the high-grade
fluid models \cite{Joseph,Lin2,MNR} and the Ladyzhenskaya models
\cite{Ladyzhenskaya}. There is an important difference between the
system (\ref{a1}) considered here and all the models mentioned
above, namely, the system (\ref{a1}) is an only partially
dissipative system.  This brings extra difficulties in the usual
existence results for small data global solutions.

There also exists a vast literature in the study of compressible
nonlinear elasticity \cite{Agemi,Sideris1} and nonlinear wave
equations
\cite{Alinhac2,Alinhac3,Christodoulou,Klainerman2,Klainerman4,Sideris2}.
The powerful techniques, the generalized energy methods, which
involve the rotation, Lorentz and scaling invariance, were
originally developed by John and Klainerman for studying the
solutions to nonlinear wave equations \cite{Klainerman1}. The
method was later generalized by Klainerman and Sideris to the
nonrelativistic wave equations and elasticity equations with a
smaller number of generators, with the absence of the Lorentz
invariance \cite{Klainerman4,Sideris1}. However, in the case of
viscoelasticity,  the presence of the viscosity term $\Delta v$
prevents the system from possessing the scaling invariant
properties. Moreover, the incompressibility is in direct violation
of the Lorentz invariant properties \cite{Sideris3,Sideris4}.

In the compressible nonlinear elasticity, the special null
condition on the energy functional $W(F)$ (or the nonlinear term
in the nonlinear wave equations) has to be imposed to carry out
the dispersive estimates for the classical solutions. Due to the
presence of the viscosity term $\Delta v$, no attempt has been
made in this paper to establish the dispersive estimates or to
understand the nonlinear wave interaction/cancellations using the
null conditions in these cases as those in
\cite{Sideris1,Sideris3} (although they are under investigation).
In fact, we use a kind of standard energy estimate as those used
for the Navier-Stokes equations. The methods in this paper are the
higher order energy estimates, which take advantages of the
presence of the dissipative term $\Delta v$ in the momentum
equation and do not take into account of the null conditions on
the elastic energy function $W(F)$. However, due to the absence of
the damping mechanism in the transport equation of $F$, we have to
use some special treatment which involves the revealing of the
special physical structures of the system. Notice that the usual
energy method \cite{Kawashima,Lei1} does not yield the small data
global existence, since there is no dissipation on the deformation
tensor $F$. Motivated by the basic energy law (see the next
section) and our earlier work in 2-D cases \cite{Lin3,Lei3}, we
analyze the induced stress term. After the usual expansion around
the equilibrium, we notice that $\nabla\cdot F$ does provide some
weak dissipation.

The other key ingredient in this paper is the observation that
$\nabla\times F$ is a high order term for initial data under our
physical considerations.  Formally, this is merely the statement
that the Lagrangian partial derivatives commute. Lemma \ref{lem23}
demonstrates the validity of this result in the evolution dynamics
of the PDE system.

The small data global existence  of the classical solutions for
the incompressible viscoelastic system (\ref{a1}) provides us
better physical understanding of this general system. The proof of
the theorems involves all the special coupling between the
transport and the induced stress, the incompressibility and the
near equilibrium expansion. Moreover, the bounds for the initial
data (which depends on the viscosity) may also shed some lights on
the large Weinssenberg number problem in viscoelasticity.

As for the related work on the existence of solutions to nonlinear
elastic (without viscosity) systems, there are works by Sideris
\cite{Sideris1} and Agemi \cite{Agemi} on the global existence of
classical small solutions to 3-D compressible elasticity  under
the assumption that the nonlinear terms satisfy the null
conditions. The former utilized the  generalized energy method
together with the additional weighted $L^2$ estimates, while the
latter's proof relies on the direct estimations of the fundamental
solutions. The global existence for 3-D incompressible elasticity
was then proved via the incompressible limit method
\cite{Sideris3} and very recently by a different method
\cite{Sideris4}. It is worth noticing that they used an Eulerian
description of the problem, which is equivalent to that in
\cite{Liu,Lin3}. Global existence for the corresponding 2-D
problem is still open,  and the related sharpest results can be
viewed in \cite{Alinhac2,Alinhac3}. For incompressible
viscoelastic fluids, Lin, Liu and Zhang \cite{Lin3} proved the
global existence in 2-D case, by introducing an auxiliary vector
field as the replacement of the transport variable $F$. Their
procedure illustrates the intrinsic nature of weak dissipation of
the induced stress tensor. Lei and Zhou \cite{Lei3} obtained the
same results via the incompressible limit where they directly
worked on the deformation tensor $F$. Recently Lei, Liu and Zhou
\cite{Lei2} proved global existence for 2-D small strain
viscoelasticity, without assumptions on the smallness of the
rotational part of the initial deformation tensor. Finally, after
the completion of this paper, we became aware  of the manuscript
\cite{ChZh} which studied the similar problems as in this paper.

The paper is organized as follows. In section 2, we review some of
the basic concepts  in mechanics. Some  important properties in
both fluid and elastic mechanics will also be presented. Section 3
is devoted to proving local existence. The proof of global
existence is completed in section 4. In section 5, the
incompressible limit is studied. The result may be important for
the study of numerical simulations and other engineering
applications.

\section{Basic Mechanics of Viscoelasticity}\label{sec1}

In this section, we will explore some of the intrinsic properties
of the viscoelastic system presented at the beginning of the
paper. These properties reflect the underlying physical origin of
the problem and in the meantime, are essential to the proof of the
global existence result here.

We recall the definition of the deformation tensor $F$. The
dynamics of any mechanical problem (under a velocity field), no
matter in fluids or solids, can be described by the flow map, a
time dependent family of orientation preserving diffeomorphisms
$x(t, X),\ 0 \leq t \leq T$. The material point (labelling) $X$ in
the reference configuration is deformed to the spatial position
$x(t, X)$ at time $t$, which is in the observer's coordinate.

The velocity field $v(t, x)$ determines the flow map, hence the
whole dynamics. However, in order to describe the changing of any
configuration or patterns during such dynamical processes, we need
to define the deformation tensor $\widetilde{F}(t, X)$:
\begin{equation}\label{c1}
\widetilde{F}(t, X) = \frac{\partial x}{\partial X}(t, X).
\end{equation}
Notice that this quantity is defined in the Lagrangian material
coordinate. Obviously it satisfies the following rule
\cite{Gurtin}:
\begin{equation}\label{c1a}
\frac{\partial \widetilde{F}(t, X)}{\partial t} = \frac{\partial
v\big(t, x(t, X)\big)}{\partial X}.
\end{equation}
In the Eulerian coordinate, the corresponding deformation tensor
$F(t, x)$ will be defined as $F(t, x(t, X)) = \widetilde{F}(t,
X)$. The equation (\ref{c1a}) will be accordingly transformed into
the third equation in system (\ref{a1}) through the chain rule
\cite{Larson,Gurtin,Liu}. In the context of the system, it can
also be interpreted as the consistency of the flow maps generated
by the velocity field $v$ and deformation field $F$.

The difference between fluids and solids lies in the fact that in
fluids, the internal energy can be determined solely by the
determinant part of $F$ (through density) and in elasticity, the
energy depends on the whole $F$.

The incompressibility can be exactly represented as
\begin{equation}
\det F = 1.
\end{equation}
The usual incompressible condition $\nabla\cdot v = 0$, the first
equation in (\ref{a1}), is the direct consequence of this
identity.

\par Since we are interested in small solutions, we define the
usual strain tensor in the form of
\begin{equation}
E = F - I.
\end{equation}

The following lemma is well known and appeared in \cite{Teman}. It
illustrates the incompressible consistence of the the system
(\ref{a1}).\\

\smallskip
\begin{lemma}\label{lem21}
Assume that the second equality of (\ref{a3}) is satisfied and
$(v, F)$ is the solution of system (\ref{a1}). Then the following
is always true:
\begin{equation}\label{b1}
\det (I + E) = 1
\end{equation}
for all time $t \geq 0$.
\end{lemma}

\begin{proof}
     Using the  identity
     $\frac{\partial\det F}{\partial F} = (\det F)F^{- T}$,
     the first and third equations of (\ref{a1}) give the result
     \begin{eqnarray}\nonumber
     &&\big(\det(I + E)\big)_t + v\cdot\nabla\big(\det(I + E)\big)\\\nonumber
     && =\det(I + E)(I + E)^{- 1}_{ji}
     \nabla_kv_i(I + E)_{kj}\\\nonumber
     &&=\det(I + E)\nabla\cdot v = 0.
     \end{eqnarray}
     Thus, the proof of Lemma \ref{lem21} is completed.
\end{proof}

The following lemma played a crucial rule in our earlier work
\cite{Lin3,Liu}. It provides the third equation in (\ref{a1}) with
a div-curl structure of compensate compactness \cite{Liu}, as that
in the vorticity equation of 3-D incompressible Euler
equations.\\

\smallskip
\begin{lemma}\label{lem22}
Assume that the third equality of (\ref{a3}) is satisfied, then
the solution $(v, F)$ of the system (\ref{a1}) satisfies the
following identities:
\begin{equation}\label{b4}
        \nabla\cdot F^T = 0, \quad{\rm and} \quad \nabla\cdot E^T = 0,
\end{equation}
for all time $t \geq 0$.
\end{lemma}

\begin{proof}
Following \cite{Lin3,Liu}, we transpose the third equation of
(\ref{a1}) and then
     apply the divergence operator to the resulting equation to
     yield
     \begin{eqnarray*}
     (\nabla_jF_{ji})_t &+& v\cdot\nabla(\nabla_jF_{ji}) +
\nabla_jv\cdot\nabla F_{ji}
     \\&=&
     \nabla_j\nabla_kv_jF_{ki} + \nabla_kv_j\nabla_jF_{ki}.
     \end{eqnarray*}
     Using the first equation of system (\ref{a1}), we obtain
     $$(\nabla_jF_{ji})_t + v\cdot\nabla(\nabla_jF_{ji}) = 0.$$
     Thus, the proof Lemma \ref{lem22} is completed.
\end{proof}

\smallskip\noindent
\begin{remark}\label{rem21}
We can derive the general form of the above identity $\nabla\cdot
F^T = 0$ from the definition of $\widetilde{F}(t, X)$ in
(\ref{c1}). However, the proof of the above lemma gives the
consistency of the system. The two algebraic identities
$\partial_{X_j}\frac{\partial\det\widetilde{F}}{\partial\widetilde{F}_{ij}}
= 0$ and $\frac{\partial\det\widetilde{F}}{\partial\widetilde{F}}
= (\det\widetilde{F})\widetilde{F}^{- T}$ give the result of
\begin{equation}\label{c3}
\partial_{X_j}(\det\widetilde{F}\widetilde{F}^{- T}_{ij}) = 0.
\end{equation}
Hence we obtain the following constraint on the deformation tensor
$F$:
\begin{eqnarray}\nonumber
\nabla_j\big[\frac{1}{\det F}F^T_{ij}\big] &=& F^{- T}_{jk}(t,
x)\partial_{X_k}\big[\frac{1}{\det
\widetilde{F}}\widetilde{F}^T_{ij}\big(t, X(t,
x)\big)\big]\\\nonumber &=&
\frac{1}{(\det\widetilde{F})}\det\widetilde{F}\widetilde{F}^{-
T}_{jk}\big(t, X(t, x)\big)
\partial_{X_k}\big[\frac{1}{\det \widetilde{F}}\widetilde{F}^T_{ij}\big(t,
X(t, x)\big)\big]\\\nonumber &=&
\frac{1}{\det\widetilde{F}}\partial_{X_k} \big[\widetilde{F}^{-
T}_{jk}\big(t, X(t, x)\big)\widetilde{F}^T_{ij}\big(t, X(t,
x)\big)\big] = 0.
\end{eqnarray}
\end{remark}

The key ingredient of the later proof in this paper is contained
in the following Lemma. It shows that $\nabla\times E$ is of
higher order.

\smallskip\noindent
\begin{lemma}\label{lem23}
Assume that the last equality of (\ref{a3}) is satisfied and $(v,
F)$ is the solution of system (\ref{a1}). Then the following
identity
\begin{equation}\label{b5}
\nabla_mE_{ij} - \nabla_jE_{im} = E_{lj}\nabla_lE_{im} -
E_{lm}\nabla_lE_{ij},
\end{equation}
holds for all time $t \geq 0$.
\end{lemma}

\begin{proof}
To prove the lemma, we will establish the evolution equation for
the quantity $\nabla_mE_{ij} - \nabla_jE_{im} -
E_{lj}\nabla_lE_{im} + E_{lm}\nabla_lE_{ij}$.

First, by the third equation of (\ref{a1}), we can
    get
    \begin{eqnarray}\label{b6}
          \partial_t\nabla_mE_{ij} &+& v\cdot\nabla\nabla_mE_{ij} +
\nabla_mv\cdot\nabla
          E_{ij}\\\nonumber
          &=& \nabla_m\nabla_kv_iE_{kj} + \nabla_kv_i\nabla_mE_{kj} +
\nabla_m\nabla_jv_i.
    \end{eqnarray}
    Thus, we have
\begin{eqnarray}\label{b7}
&&\partial_t\big(\nabla_mE_{ij} - \nabla_jE_{im}\big) +
  v\cdot\nabla\big(\nabla_mE_{ij} - \nabla_jE_{im}\big)\\\nonumber
&&\quad +\ \big(\nabla_mv\cdot\nabla E_{ij} - \nabla_jv\cdot\nabla
  E_{im}\big)\\\nonumber
&&= \big(\nabla_m\nabla_kv_iE_{kj} -
  \nabla_j\nabla_kv_iE_{km}\big)\\\nonumber
&&\quad +\ \nabla_kv_i\big(\nabla_mE_{kj} - \nabla_jE_{km}\big).
    \end{eqnarray}
On the other hand, combining (\ref{b6}) and the third equation of
(\ref{a1}), we have
\begin{eqnarray}\nonumber
&&\partial_t\big(E_{lm}\nabla_lE_{ij}\big) +
v\cdot\nabla\big(E_{lm}\nabla_lE_{ij}\big)\\\nonumber &&=
\nabla_l\big[\nabla_kv_iE_{kj}E_{lm} + \nabla_jv_iE_{lm} +
\nabla_mv_lE_{ij}\big].
\end{eqnarray}
     Thus, we get
\begin{eqnarray}\label{b8}
    &&\partial_t\big(E_{lm}\nabla_lE_{ij} - E_{lj}\nabla_lE_{im}\big)
       + v\cdot\nabla\big(E_{lm}\nabla_lE_{ij} -
       E_{lj}\nabla_lE_{im}\big)\\\nonumber
    &&   = \nabla_l\big[\nabla_kv_i\big(E_{kj}E_{lm} - E_{km}E_{lj}\big)
       + \big(\nabla_jv_iE_{lm} - \nabla_mv_iE_{lj}\big)\\\nonumber
     &&\quad +\ \big(\nabla_mv_lE_{ij} - \nabla_jv_lE_{im}\big)\big].
\end{eqnarray}

Combing (\ref{b7}) and (\ref{b8}), we obtain
\begin{eqnarray}\nonumber
&&\partial_t\big(\nabla_mE_{ij} - \nabla_jE_{im}
  + E_{lm}\nabla_lE_{ij} - E_{lj}\nabla_lE_{im}\big)\\\nonumber
&&\quad +\ v\cdot\nabla\big(\nabla_mE_{ij} - \nabla_jE_{im} +
  E_{lm}\nabla_lE_{ij} - E_{lj}\nabla_lE_{im}\big)\\\nonumber
&&= \nabla_l\big[\nabla_kv_i\big(E_{kj}E_{lm} -
E_{km}E_{lj}\big)\big] +\ \nabla_kv_i\big(\nabla_mE_{kj} -
\nabla_jE_{km}\big)
\end{eqnarray}
Using the first equation of (\ref{a1}), we
    have
    \begin{eqnarray}\nonumber
       &&\partial_t\big(\nabla_mE_{ij} - \nabla_jE_{im}
           + E_{lm}\nabla_lE_{ij} -
           E_{lj}\nabla_lE_{im}\big)\\\nonumber
       &&\quad +\ v\cdot\nabla\big(\nabla_mE_{ij} -
       \nabla_jE_{im} + E_{lm}\nabla_lE_{ij} -
     E_{lj}\nabla_lE_{im}\big)\\\nonumber
       &&= \nabla_l\big[\nabla_ku_i\big(E_{kj}E_{lm} - E_{km}E_{lj}\big)
           + v_i\big(\nabla_mE_{lj} - \nabla_jE_{lm}\big)\big]
    \end{eqnarray}
On the other hand, noting (\ref{b4}), this gives
    \begin{eqnarray}\nonumber
       &&\partial_t\big(\nabla_mE_{ij} - \nabla_jE_{im}
           + E_{lm}\nabla_lE_{ij} -
           E_{lj}\nabla_lE_{im}\big)\\\nonumber
       &&\quad +\ v\cdot\nabla\big(\nabla_mE_{ij} -
       \nabla_jE_{im} + E_{lm}\nabla_lE_{ij} -
       E_{lj}\nabla_lE_{im}\big)\\\nonumber
       &&= \nabla_l\big[- \big(v_iE_{kj}\nabla_kE_{lm} -
v_iE_{km}\nabla_kE_{lj}\big)
           + v_i\big(\nabla_mE_{lj} - \nabla_jE_{lm}\big)\big]\\\nonumber
&&\quad +\ \nabla_l\big[\nabla_k\big(v_iE_{kj}E_{lm} -
v_iE_{km}E_{lj}\big)\big]\\\nonumber
       &&= \nabla_l\big[v_i\big(\nabla_mE_{lj} - \nabla_jE_{lm}
           + E_{km}\nabla_kE_{lj} -
           E_{kj}\nabla_kE_{lm}\big)\big]\\\nonumber
       &&\quad +\ \nabla_l\nabla_k\big(v_iE_{kj}E_{lm}\big)
           - \nabla_l\nabla_k\big(v_iE_{km}E_{lj}\big)\\\nonumber
  &&= \nabla_l\big[v_i\big(\nabla_mE_{lj} - \nabla_jE_{lm}
           + E_{km}\nabla_kE_{lj} - E_{kj}\nabla_kE_{lm}\big)\big]
\end{eqnarray}
At last, by using (\ref{b4}) and the first equation of (\ref{a1})
once again, we get the evolution of the concerned quantity.
       \begin{eqnarray}\nonumber
       &&\partial_t\big(\nabla_mE_{ij} - \nabla_jE_{im}
           + E_{lm}\nabla_lE_{ij} -
           E_{lj}\nabla_lE_{im}\big)\\\nonumber
       &&\quad +\ v\cdot\nabla\big(\nabla_mE_{ij} -
       \nabla_jE_{im} + E_{lm}\nabla_lE_{ij} -
       E_{lj}\nabla_lE_{im}\big)\\\nonumber
       &&= \nabla_lv_i\big(\nabla_mE_{lj} - \nabla_jE_{lm}
           + E_{km}\nabla_kE_{lj} -
           E_{kj}\nabla_kE_{lm}\big)\\\nonumber
           &&\quad +\ v_i\big(\nabla_lE_{km}\nabla_kE_{lj} -
        \nabla_lE_{kj}\nabla_kE_{lm}\big)\\\nonumber
       &&= \nabla_lv_i\big(\nabla_mE_{lj} - \nabla_jE_{lm}
          + E_{km}\nabla_kE_{lj} - E_{kj}\nabla_kE_{lm}\big).
    \end{eqnarray}
During the calculations, we use the incompressibility conditions
(\ref{b4}) and the first equation of (\ref{a1}) in the second, the
third and the sixth equality. The last equality proves the lemma,
since the above quantity will maintain zero all the time with zero
initial condition.
\end{proof}

\smallskip\noindent
\begin{remark}\label{rem22}
In order to demonstrate the mechanical background of the above
lemma, we again go back to  the definition of $\widetilde{F}(t,
X)$ in (\ref{c1}). Formally, the fact that the Lagrangian
derivatives commute yields the fact that
$\partial_{X_k}\widetilde{F}_{ij} =
\partial_{X_j}\widetilde{F}_{ik}$, which is equivalent to
$\widetilde{F}_{lk}\nabla_lF_{ij}(t, x(t, X)) =
\widetilde{F}_{lj}\nabla_lF_{ik}(t, x(t, X)).$ Thus, one has
$$F_{lk}\nabla_lF_{ij}(t, x) =
F_{lj}\nabla_lF_{ik}(t, x)$$ which means that
$$\nabla_kE_{ij} + E_{lk}\nabla_lE_{ij} = \nabla_jE_{ik} +
E_{lj}\nabla_lE_{ik}(t, x)$$ This is exactly the result in the
above lemma. However, the validity of the statement for any
solution of the system (\ref{a1}) is the merit of the above lemma.
\end{remark}

Finally, we make some simplifications for system (\ref{a1}).

In addition to their definitions as the elementary symmetric
functions of the eigenvalues, the invariants $\gamma(A)$ of any $3
\times 3$ matrix $A$ are conveniently expressed as
$$\gamma_1(A) = trA,\ \ \gamma_2(A)
= \frac{1}{2}\big[(trA)^2 - trA^2\big],\ \ \gamma_3(A) = \det A.$$
On the other hand, one can easily get the identity
$$\gamma_3(A + I) = 1 + \gamma_1(A) + \gamma_2(A) + \gamma_3(A).$$

Combining the above identity with (\ref{b1}), one can obtain the
incompressible constraint on $E$ as
\begin{equation}\label{b2}
trE = - \det E - \gamma_2(E).
\end{equation}
By a similar process, the incompressible constraint on $E$ in
2-dimension case takes
\begin{equation}\label{b3}
trE = - \det E.
\end{equation}

 Next, we will consider the isotropic strain energy function $W(F)$.
We let $f_1(E)$,  $f_2(E)$ and $f_3(E)$ represent any generic
terms of degree two or higher at the origin. In the isotropic
case, $W$ depends on $F$ through the principal invariants of the
strain matrix $FF^T$ \cite{Gurtin}. Define the linearized
elasticity tensor as
\begin{equation}\label{e1}
A^{ij}_{lm} = \frac{\partial^2W}{\partial F_{il}F_{jm}}(I).
\end{equation}

Suppose that the strain energy function $W(F)$ is isotropic and
frame indifferent, the strong Legendre-Hadamard ellipticity
condition imposed upon the linearized elasticity tensor (\ref{e1})
takes the form of:
\begin{equation}\label{e2}
A^{ij}_{lm} = (\alpha^2 - 2\beta^2)\delta_{il}\delta_{jm} +
\beta^2(\delta_{im}\delta_{jl} + \delta_{ij}\delta_{lm}),\quad
{\rm with}\ \alpha > \beta > 0,
\end{equation}
where the positive parameters $\alpha$ and $\beta$ depend only on
$W$. They represent the speeds of propagation of pressure and
shear waves, respectively. By (\ref{b4}), (\ref{e1}) and
(\ref{e2}), we have
\begin{eqnarray}\label{e3}
\nabla_l\big[\frac{\partial W(F)}{\partial F}F^T\big]_{il}
  &=& \nabla_l\big[\frac{\partial W(F)}{\partial F}E^T\big]_{il}
      + \nabla_l\frac{\partial W(F)}{\partial F_{il}}\\\nonumber
  &=& \nabla_l\big[\frac{\partial W(F)}{\partial F}E^T\big]_{il}
      + \frac{\partial^2W(I)}{\partial F_{il}\partial
F_{jm}}\nabla_lE_{jm}
      + \nabla_l f_1(E)_{il}\\\nonumber
  &=& (\alpha^2 - 2\beta^2)\nabla_i trE + \beta^2\big(\nabla\cdot E^T +
\nabla\cdot E\big)_i
      + \nabla_l f_2(E)_{il}\\\nonumber
  &=& \beta^2(\nabla\cdot E)_i + \nabla_l f_3(E)_{il},
\end{eqnarray}
where we also used the assumptions that the reference
configuration is a stress-free state:
\begin{equation}\label{w}
\frac{\partial W(I)}{\partial F} = 0.
\end{equation}
Without loss of generality, we assume that the constant $\beta =
1$. In particular, in what follows, we only consider the case of
the Hookean elastic materials: $\nabla\cdot f_3(E) = \nabla\cdot
(EE^T)$. The system is
\begin{equation}\label{e4}
\left\{
  \begin{array} {l@{\quad \ \quad}l}
  \nabla \cdot v = 0,\\
  v_t^i + v \cdot \nabla v^i + \nabla_i p
    = \mu\Delta v^i + E_{jk}\nabla_j E_{ik} + \nabla_jE_{ij},\\
    E_t + v \cdot \nabla E = \nabla vE + \nabla v.
  \end{array}
\right.
\end{equation}
All the following proofs and results are also valid for general
isotropic elastic energy functions satisfying the strong
Legendre-Hadamard ellipticity condition, as those in (\ref{e3}).

\section{Local Existence}

Although the proof of the following local existence theorem is
lengthy, the idea is straight forward and had been carried out in
the case of 2-D Hookean elasticity in \cite{Lin3}. For a
self-contained presentation, we will carry out the similar proofs
into our general cases.

\smallskip
\begin{theorem}\label{thm31}
Let $k \geq 2$ be a positive integer, and $v_0,\ E_0 \in
H^k(\Omega)$ which satisfy the incompressible constraint
(\ref{a3}). Suppose that the isotropic elastic energy function
satisfies the constitutive assumption (\ref{e2}). Then there
exists a positive time $T$, which depends only on
$\|v_{0}\|_{H^{2}}$ and $\|E_{0}\|_{H^{2}}$, such that the initial
value problem or the periodic initial-boundary value problem for
(\ref{a1}) (or (\ref{e4})) has a unique classical solution in the
time interval $[0, T)$ which satisfies
\begin{equation}\label{f1}
\left\{
  \begin{array} {l@{\quad \ \quad}l}
  \partial_t^{j}\nabla^{\alpha}v \in L^{\infty}\big(0, T; H^{k - 2j -
        |\alpha|}(\Omega)\big) \cap L^{2}\big(0, T; H^{k - 2j - |\alpha| +
1}(\Omega)\big),\\
  \partial_t^{j}\nabla^{\alpha}E \in L^{\infty}\big(0, T; H^{k - 2j -
        |\alpha|}(\Omega)\big).
  \end{array}
\right.
\end{equation}
for all $j$, $\alpha$ satisfying $2j + |\alpha| \leq k$. Moreover,
if $T^{\star} < + \infty$ is the lifespan of the solution, then
\begin{equation}\label{f2}
\int_{0}^{T^{*}}\|\nabla v\|_{H^{2}}^{2}\ dt = + \infty.
\end{equation}
\end{theorem}

\begin{proof} By the  Galerkin's method originally for standard
Navier-Stokes equation \cite{Teman} and later modified for
different coupling system \cite{Lin1}, we can construct the
approximate solutions to the momentum equation of $v$, and then
substitute this approximate $v$ into the transport equation to get
the appropriate solutions of $E$. To prove the convergence of the
sequence consisting of the approximate solutions, we need only a
$priori$ estimates for them. For simplicity, we will establish a
$priori$ estimates for the smooth solutions of (\ref{e4}).
Therefore, let us assume in the rest of this section that $(v, E)$
is a local smooth solution to system (\ref{e4}) on some time
interval $[0, T)$.

In this paper, $\|\cdot\|$ will denote the $L^2(\Omega)$ norm,
where $\Omega \subseteq R^n$ will be either an n-dimensional torus
$T^n$, or the entire space $R^n$ for $n = 2$ or 3, and $(\cdot,
\cdot)$ the inner product of standard space $L^2(\Omega)^d$ with
$d \in \{1, 2, 3, 4, 9\}$.

The original system (\ref{a1}) possesses the following energy law:
\begin{equation}\label{f3a}
\frac{d}{dt}\Big(\frac{1}{2}\|v\|^2 + \int_\Omega
\big(W(F)-W(I)\big) \, dx\Big) + \mu\|\nabla v\|^2 = 0.
\end{equation}
Equivalently, for (\ref{e4}), the corresponding energy law will
be:
\begin{equation}\label{f3}
\frac{1}{2}\frac{d}{dt}\big(\|v\|^2 + \|E\|^2\big) + \mu\|\nabla
v\|^2 = 0,
\end{equation}
which follows from the third equation of (\ref{e4}) and the
incompressibility.

The following well-known interpolation inequalities are results of
the Sobolev embedding theorems and scaling techniques
\cite{Alinhac1,Lin1}. They will be frequently used in the
following higher order energy estimates.\\

\begin{lemma}\label{lem31}
Assume $v \in W^{k,2}(\Omega), k \geq 3$. The following
interpolation inequalities hold.
\begin{enumerate}
     \item For $1 \leq s \leq k$,
          \begin{eqnarray*}
               \|v\|_{L^{4}} &\leq& C\|v\|^{1 - \frac{1}{2s}}
\|\nabla^{s} v\|^{\frac{1}{2s}},
               \ \ \Omega \subseteq R^{2},\\
               \|v\|_{L^{4}} &\leq& C\|v\|^{1 - \frac{3}{4s}}
\|\nabla^{s} v\|^{\frac{3}{4s}},
               \ \ \Omega \subseteq R^{3},\\
               \|\nabla v\|_{L^{4}} &\leq& C\|v\|^{1 - \frac{3}{2(s +
1)}} \|\nabla^{s}\nabla v\|^
               {\frac{3}{2(s + 1)}},\ \ \Omega \subseteq R^{2},\\
               \|\nabla v\|_{L^{4}} &\leq& C\|v\|^{1 - \frac{7}{4(s +
1)}} \|\nabla^{s}\nabla v\|^
               {\frac{7}{4(s + 1)}},\ \ \Omega \subseteq R^{3},\\
               \|\Delta v\|_{L^{4}} &\leq& C\|v\|^{1 - \frac{5}{2(s +
2)}}\|\nabla^s\Delta v\|^
               {\frac{5}{2(s + 2)}},\ \ \Omega \subseteq R^{2},\\
               \|\Delta v\|_{L^{4}} &\leq& C\|v\|^{1 - \frac{11}{4(s +
2)}}\|\nabla^s\Delta v\|^
               {\frac{11}{4(s + 2)}},\ \ \Omega \subseteq R^{3},
          \end{eqnarray*}
     \item For $2 \leq s \leq k$,
          \begin{eqnarray*}
               \|v\|_{L^{\infty}} &\leq& C\|v\|^{1 - \frac{1}{s}}
\|\nabla^{s} v\|^{\frac{1}{s}},
               \ \ \Omega \subseteq R^{2},\\
               \|v\|_{L^{\infty}} &\leq& C\|v\|^{1 - \frac{3}{2s}}
\|\nabla^{s} v\|^{\frac{3}{2s}},
               \ \ \Omega \subseteq R^{3},\\
               \|\nabla v\|_{L^{\infty}} &\leq& C\|v\|^{1 - \frac{2}{s
+ 1}}\|\nabla^s\nabla v\|^
               {\frac{2}{s + 1}},\ \ \Omega \subseteq R^{2},\\
               \|\nabla v\|_{L^{\infty}} &\leq& C\|v\|^{1 -
\frac{5}{2(s + 1)}}\|\nabla^s\nabla
               v\|^{\frac{5}{2(s + 1)}},\ \ \Omega \subseteq R^{3}.
          \end{eqnarray*}
\end{enumerate}
\end{lemma}

The following two propositions can be found in
\cite{Alinhac1,Klainerman3,Lin1}.\\

\smallskip
\begin{proposition}\label{prop31}
If $g: R^n \longrightarrow R$ is a smooth function with $g(0) =
0$, then, for any positive constant $k$, $g(v) \in L^\infty \cap
H^k$ if $v \in L^\infty \cap H^k$ and
\begin{equation}\nonumber
\|g(v)\|_{H^k} \leq C\|v\|_{H^k}
\end{equation}
for some constant $C$ depending only on $g$, $k$ and
$\|v\|_{L^\infty}$.
\end{proposition}

\smallskip
\begin{remark}\label{rem31}
The above proposition is only used in the cases of general elastic
energy functions.
\end{remark}

\smallskip
\begin{proposition}\label{prop32}
Assume that $f,\ g \in H^s(\Omega)$. Then for any multi-index
$\alpha,\ |\alpha| \leq s$, we have
\begin{equation}\label{f4}
      \left\{\begin{array}{l@{\quad\ \quad}l}
  \|\nabla^\alpha(fg)\| \leq C\big(\|f\|_{L^\infty}\|\nabla^s g\| +
\|g\|_{L^\infty}\|\nabla^sf\|\big),\\
  \|\nabla^\alpha(fg) - f\nabla^\alpha g\| \leq C\big(\|\nabla
f\|_{L^\infty}\|\nabla^{s - 1}g\|
  + \|\nabla g\|_{L^\infty}\|\nabla^{s - 1}f\|\big).
\end{array}\right.
\end{equation}
for some constant $C$ depending only on $n$.
\end{proposition}

We divide the proof of the theorem \ref{thm31} into two parts.

\textbf{Step 1.} $H^{2}$ estimate.

Integrate (\ref{f3}) over $(0, t)$, one obtains
\begin{equation}\label{g1}
\big(\|v\|^{2} + \|E\|^{2}\big) + 2\mu\int_{0}^t\|\nabla v\|^{2}\
dt = \big(\|v_{0}\|^{2} + \|E_{0}\|^{2}\big).
\end{equation}
\par By taking the $L^{2}$ inner product of
the second equation in (\ref{e4}) with $\Delta v$, using Lemma
\ref{lem31} and integration by parts, we have
\begin{eqnarray}\nonumber
&&\mu\|\Delta v\|^{2}\\\nonumber &&= (v_t, \Delta v) + (v \cdot
  \nabla v, \Delta v) + (\nabla p, \Delta v)\\\nonumber
&&\quad -\ (E_{jk}\nabla_j E_{ik}, \Delta
  v^{i}) - (\nabla\cdot E, \Delta v)\\\nonumber
&&\leq C\|\Delta v\|\Big(\|v_t\| + \|v\|_{L^{\infty}}\|\nabla v\|
  + \|E\|_{L^{\infty}}\|\nabla E\| + \|\nabla E\|\Big)\\\nonumber
&&\leq C\|\Delta v\|\Big(\|v_t\| + \|v\|^{1 - \theta(2)}\|\Delta
  v\|^{\theta(2)}\|\nabla v\|\\\nonumber
&&\quad +\ \big(\|E\|^{1 - \theta(2)}\|\Delta E\|^{\theta(2)} +
  1\big)\big(\|E\| + \|\Delta E\|\big)\Big)\\\nonumber
&&\leq \frac{1}{2}\mu\|\Delta v\|^{2} + g\big(\|v_t\|, \|\nabla
  v\|, \|\Delta E\|\big).
\end{eqnarray}
where $\theta(s)$ $(0 < \theta(s) < 1)$ represents a generic
function which is determined by Lemma \ref{lem31}, and $g(\cdot,
\cdot, \cdot,)$ represents any generic nonnegative and increasing
function of its variables. Thus, we have
\begin{equation}\label{g2}
\|\Delta v\|^{2} \leq g\big(\|v_t\|, \|\nabla v\|, \|\Delta
E\|\big).
\end{equation}

In the meantime, by taking the $L^{2}$ inner product of the second
equation in (\ref{e4}) with $v_t$, using Lemma \ref{lem31} and
integration by parts, we obtain

\begin{eqnarray}\label{g3}
&&\frac{\mu}{2}\frac{d}{dt}\|\nabla v\|^{2} +
  \|v_t\|^{2}\\\nonumber
&&= -(v \cdot \nabla v, v_t) - (\nabla p, v_t) + (E_{jk}\nabla_j
E_{ik}, v_t^i) + (\nabla\cdot E, v_t)\\\nonumber &&\leq \|\nabla
v_t\|\big(\|E\|_{L^{4}}^{2} + \|v\|_{L^{4}}^{2} +
  \|E\|\big) - (E_{ik}\nabla_jE_{jk}, v_t^i)\\\nonumber
&&\leq C\|\nabla v_t\|\big(\|E\|^{1 - \theta(2)}\|\Delta
  E\|^{\theta(2)} + \|v\|^{1 - \theta(1)}\|\nabla v\|^{\theta(1)} +
  \|E\|\big)\\\nonumber
&&\leq \frac{\mu}{8}\|\nabla v_t\|^{2} + g\big(\|v_t\|, \|\nabla
  v\|, \|\Delta E\|\big).
\end{eqnarray}
In order to get the first inequality of the above computation, we
used the constraint on $E$, which is due to the incompressibility,
in Lemma \ref{lem22}.

Next, taking $t$ derivative of the second equation in (\ref{e4}),
and then taking the $L^{2}$ inner product of the resulting
equation with $v_t$, we can apply Lemma \ref{lem31} and
integration by parts to obtain
\begin{eqnarray}\label{g4}
&&\frac{1}{2}\frac{d}{dt}\|v_t\|^{2} + \mu \|\nabla
  v_t\|^{2}\\\nonumber
&&= -\big(\partial_t(v \cdot \nabla v), v_t\big)
  - (\nabla p_t, v_t) + \big(\partial_t(E_{kj}\partial_{k}E_{ij}), v_t^{i}\big)
  + (\partial_t\partial_{j}E_{ij}, v_t^{i})\\\nonumber
&&= -(v_t \cdot \nabla v, v_t)
  - \big(\partial_t(E_{kj}E_{ij}),\partial_{k} v_t^{i}\big) -
  (\partial_tE_{ij}, \partial_{j}v_t^{i})\\\nonumber
&&= (v \otimes v_t, \nabla v_t)
  - \big(\partial_t(EE^T),\nabla v_t\big)
  - (\partial_tE, \nabla v_t)\\\nonumber
&&\leq \|\nabla v_t\|\big(\|v_t\|\|v\|_{L^{\infty}} +
  \|E_t\|\|E\|_{L^{\infty}} +  \|E_t\|\big)\\\nonumber
&&\leq \|\nabla v_t\|\big(\|v_t\|\|v\|^{1 - \theta(2)}\|\Delta
  v\|^{\theta(2)}\\\nonumber
&&\quad +\ \|E_t\|\|E\|^{1 - \theta(2)}\|\Delta E\|^{\theta(2)} +
  \|E_t\|\big)\\\nonumber
&&\leq \frac{\mu}{8}\|\nabla v_t\|^{2} + g\big(\|v_t\|, \|\Delta
  E\|, \|E_t\|, \|\Delta v\|\big).
\end{eqnarray}
On the other hand, from the transport equation of (\ref{e4}) we
have
\begin{eqnarray*}
&&\|E_t\| \leq \|E\|_{L^{\infty}}\|\nabla v\|
+ \|v\|_{L^{\infty}} \|\nabla E\| + \|\nabla v\|\\
&&\leq  C\|E\|^{1 - \theta(2)}\|\Delta E\|^{\theta(2)}\|\nabla v\|
+ \|\nabla v\|\\ &&\quad +\ \|v\|^{1 - \theta(2)}\|\Delta
v\|^{\theta(2)} \big(\|E\| + \|\Delta E\|\big)\\
&&\leq g\big(\|\nabla v\|, \|\Delta E\|, \|\Delta v\|\big).
\end{eqnarray*}
\par Substitute (\ref{g2}) into the above inequality, one has
\begin{equation}\label{g5}
\|E_t\| \leq g\big(\|v_t\|, \|\nabla v\|, \|\Delta E\|\big).
\end{equation}
Plugging (\ref{g2}) and (\ref{g5}) into (\ref{g4}), one arrives at
\begin{equation}\label{g6}
\frac{1}{2}\frac{d}{dt}\|v_t\|^{2} + \frac{7\mu}{8} \|\nabla
v_t\|^{2} \leq g\big(\|v_t\|, \|\nabla v\|, \|\Delta E\|\big).
\end{equation}
\par Noting (\ref{g3}) and (\ref{g6}), it is clear that the key now is the
estimate of  the term $\|\Delta E\|$. It follows from the
transport equation in (\ref{e4}) that
\begin{eqnarray}\label{g7}
&&\frac{1}{2}\frac{d}{dt}\|\Delta E\|^{2}\\ \nonumber&&= -
\big(\Delta (v \cdot \nabla E), \Delta E\big) + \big(\Delta
(\nabla v E), \Delta E\big) + (\nabla\Delta v , \Delta E)\\
\nonumber&&= - \big(\Delta (v \cdot \nabla E) -  v \cdot
\nabla\Delta E, \Delta E\big) + \big(\Delta (\nabla v E), \Delta
E\big) + (\nabla\Delta v , \Delta E)\\\nonumber &&\leq C\|\Delta
E\|\Big(\|\Delta E\|\|\nabla v\|_{L^{\infty}} + \|\Delta
v\|_{L^{4}}\|\nabla E\|_{L^{4}} + \|\nabla\Delta v\|\|
E\|_{L^{\infty}} + \|\nabla\Delta v\|\Big)\\\nonumber &&\leq
C\|\Delta E\|\Big(\|\Delta E\|\|\nabla\Delta v\|^{\theta(2)}\|
v\|^{1 - \theta(2)} + \|\nabla\Delta v\|\|\Delta
E\|^{\theta(2)}\|E\|^{1 - \theta(2)}\\\nonumber &&\quad +\
\|\nabla\Delta v\|^{\theta(1)}\| v\|^{1 - \theta(1)}\|\Delta
E\|^{\theta(1)}\|E\|^{1 - \theta(1)} + \|\nabla\Delta
v\|\Big)\\\nonumber &&\leq g(\|\Delta E\|)\|\nabla\Delta
v\|.\end{eqnarray}

On the other hand, by applying $\nabla$ to the momentum equation
in (\ref{e4}) and then taking the $L^{2}$ inner product of the
resulting equation with $\nabla\Delta v$, we can get
\begin{eqnarray}\nonumber
&&\mu\|\nabla\Delta v\|^{2}\\\nonumber &&= (\nabla v_t,
\nabla\Delta v) + \big( \nabla(v \cdot \nabla v), \nabla\Delta
v\big) + ( \nabla\nabla p, \nabla\Delta v)\\\nonumber &&\quad -\
\big(\nabla(E_{kj}\partial_{k}E_{ij}), \nabla\Delta v^{i}\big) -
(\nabla\partial_{j}E_{ij}, \nabla\Delta v^{i})\\\nonumber &&\leq
C\|\nabla\Delta v\|\Big(\|\nabla v_t\| + \|\Delta
v\|\|v\|_{L^{\infty}} + \|\nabla v\|_{L^{4}}^{2}\\\nonumber
&&\quad +\ \|\Delta E\|\|E\|_{L^{\infty}} + \|\nabla
E\|_{L^{4}}^{2} + \|\Delta E\|\Big)\\\nonumber  &&\leq
C\|\nabla\Delta v\|\Big(\|\nabla v_t\| + \|\Delta v\|\|\Delta
v\|^{\theta(2)}\|v\|^{1 - \theta(2)}+ \|v\|^{2 -
2\theta(1)}\|\Delta v\|^{2\theta(1)} \\\nonumber &&\quad +\
\|\Delta E\|\|\Delta E\|^{\theta(2)}\|E\|^{1 - \theta(2)} +
\|E\|^{2 - 2\theta(1)}\|\Delta E\|^{2\theta(1)} + \|\Delta
E\|\Big).
\end{eqnarray}
Using (\ref{g2}) again, it yields
\begin{equation}\label{g8}
\|\nabla\Delta v\| \leq C\|\nabla v_t\| + g\big(\|v_t\|, \|\nabla
v\|, \|\Delta E\|\big).
\end{equation}
Insert (\ref{g8}) into (\ref{g7}), one concludes that
\begin{equation}\label{g9}
\frac{1}{2}\frac{d}{dt}\|\Delta E\|^{2} \leq \frac{\mu}{8}\|\nabla
v_t\|^{2} + g\big(\|v_t\|, \|\nabla v\|, \|\Delta E\|\big).
\end{equation}
\par Combining (\ref{g3}), (\ref{g6}) with (\ref{g9}), we arrive at
\begin{eqnarray}\label{g10}
\frac{d}{dt}\big(\|\Delta E\|^{2} &+& \mu\|\nabla v\|^{2} +
\|v_t\|^{2}\big) + \big(\mu\|\nabla v_t\|^{2} +
\|v_t\|^{2}\big)\\\nonumber &\leq& g\big(\|v_t\|, \|\nabla v\|,
\|\Delta E\|\big).
\end{eqnarray}
It follows from the momentum equation in (\ref{e4}) that
\begin{equation}\label{g11}
\|v_t(0, x)\| \leq C\big(\|v_{0}\|_{H^{2}},
\|E_{0}\|_{H^{2}}\big).
\end{equation}
(\ref{g10}), (\ref{g11}) and the Gronwall's inequality guarantee
the fact that there exist positive constants $T, M_0$, depending
only on $\|v_{0}\|_{H^{2}}, \|E_{0}\|_{H^{2}}$ such that
\begin{equation}\label{g12}
\big(\|\Delta E\|^{2} + \mu\|\nabla v\|^{2} + \|v_t\|^{2}\big) +
\int_{0}^T\big(\mu\|\nabla v_t\|^{2} + \|v_t\|^{2}\big)\ ds \leq
M_0.
\end{equation}
Returning to (\ref{g2}) and (\ref{g5}), we find that
\begin{equation}\label{g13}
\|\Delta v\| \leq g(M_0),\ \ \|E_t\| \leq g(M_0).
\end{equation}
And recall (\ref{g8}), we can obtain from (\ref{g12}) that
\begin{equation}\label{g14}
\int_{0}^T\|\nabla\Delta v\|^{2}\ ds \leq g(M_0).
\end{equation}
By (\ref{g1}), (\ref{g12})-(\ref{g14}), we conclude that there
exists a sufficiently large positive constant $M$ depending only
on $\|v_{0}\|_{H^{2}}, \|E_{0}\|_{H^{2}}$ such that
\begin{equation}\label{g15}
\big(\|v\|_{H^2}^2 + \|E\|_{H^2}^2 + \|v_t\|^{2} +
\|E_t\|^{2}\big) + \int_{0}^T\big(\|\nabla v\|_{H^2}^{2} +
\|v_t\|_{H^1}^{2}\big)\ ds \leq M.
\end{equation}
We complete the proof of (\ref{f1}) when $k = 2$.

To prove (\ref{f2}), we assume that $T^{*} < \infty$ is the
maximal existence time and
\begin{equation}\label{h1}
\int_{0}^{T^{*}}\|\nabla v\|_{H^{2}}^{2}\ dt < + \infty.
\end{equation}
Go back to (\ref{g7}), we can use Gronwall's inequality to get
\begin{equation}\label{h2}
\|\Delta E\| < + \infty,\ \ \ \ \ 0 \leq t \leq T^{\star}.
\end{equation}
On the other hand, by (\ref{h2}) and the transport equation of $E$
in (\ref{e4}), we have
$$\|E_t\|^2 \leq \big(\|E\|_{L^{\infty}}\|\nabla v\|
+ \|v\|_{L^{\infty}} \|\nabla E\| + \|\nabla v\|\big)^2$$
$$\leq K + \|\nabla v\|_{H^2}^{2},\ \ \ \  \ 0 \leq t \leq T^{*}.\ \ \
\ \ $$ Thus, by (\ref{h1}), we obtain
\begin{equation}\label{h3}
\int_{0}^{T^{*}}\|E_t\|^{2}\ dt < \infty.
\end{equation}

If we  go back to (\ref{g3}), using  (\ref{g1}) and (\ref{h2}), we
have
\begin{eqnarray}\label{h4}
&&\frac{\mu}{2}\frac{d}{dt}\|\nabla v\|^{2} +
  \|v_t\|^{2}\\\nonumber
&&\leq C\|\nabla v_t\|\big(\|E\|^{1 - \theta(2)}\|\Delta
  E\|^{\theta(2)}\\\nonumber
&&\quad +\ \|v\|^{1 - \theta(1)}\|\nabla v\|^{\theta(1)} +
  \|E\|\big)\\\nonumber
&&\leq \frac{\mu}{8}\|\nabla v_t\|^2 + C\|\nabla v\|^2 + C.
\end{eqnarray}

Similarly, by (\ref{g1}) and (\ref{h2}),  (\ref{g4}) will give
\begin{eqnarray}\label{h5}
&&\frac{1}{2}\frac{d}{dt}\|v_t\|^{2} + \mu \|\nabla
  v_t\|^{2}\\\nonumber
&&\leq \|\nabla v_t\|\big(\|v_t\|\|v\|^{1 - \theta(2)}\|\Delta
 v\|^{\theta(2)} \\\nonumber
&&\quad +\ \|E_t\|\|E\|^{1 - \theta(2)}\|\Delta E\|^{\theta(2)} +
  \|E_t\|\big)\\\nonumber
&&\leq \frac{\mu}{8}\|\nabla v_t\|^{2} + C\|v_t\|^2\big(\|\Delta
  v\|^2 + 1\big) + \|E_t\|^2.
\end{eqnarray}
Combining (\ref{h4}) and (\ref{h5}), we have
$$\frac{d}{dt}\big(\mu\|\nabla v\|^{2} + \|v_t\|^{2}\big) + \|v_t\|^{2} +
\mu \|\nabla v_t\|^{2} \leq C\|v_t\|^2\big(\|\Delta v\|^2 + 1\big)
+ \|E_t\|^2.$$

With (\ref{h1}) and (\ref{h3}), we can use Gronwall's inequality
to get
\begin{equation}\label{h6}
\big(\|\nabla v\|^{2} + \|v_t\|^{2}\big) +
\int_{0}^{T^{*}}\big(\|\nabla v_t\|^{2} + \|v_t\|^{2}\big)\ dt < +
\infty.
\end{equation}
Insert (\ref{h2}) and (\ref{h6}) into (\ref{g2}) and (\ref{g5}),
we get
\begin{equation}\label{h7}
\|\Delta v\| < + \infty,\ \ \|E_t\| < + \infty,\ \ \ 0 \leq t \leq
T^{\star}.
\end{equation}
\par Combining (\ref{g1}), (\ref{h1})-(\ref{h2}) and
(\ref{h6})-(\ref{h7}), we get
$$\big(\|v\|_{H^2}^2 + \|E\|_{H^2}^2 + \|v_t\|^{2} + \|E_t\|^{2}\big)
+ \int_{0}^{T^\star}\big(\|\nabla v\|_{H^2}^{2} +
\|v_t\|_{H^1}^{2}\big)\ ds < + \infty,$$ which contradicts with
the assumption that $T^{*}$ is the maximal existence time, which
in turn proves the
equation (\ref{f2})  when $k = 2$.\\

\textbf{Step 2.} Higher order energy estimate.

The proof  for $k \geq 2$ is an induction on $k$. Assume the
theorem is valid for integer $k$. In other words, we have
\begin{equation}\label{ccc}
      \left\{\begin{array}{l@{\quad\ \quad}l}
\partial_t^{i}\nabla^{\alpha}v \in L^{\infty}\big(0, T; H^{k - 2i
- |\alpha|}(\Omega)\big) \cap L^{2}\big(0, T; H^{k - 2i - |\alpha|
+ 1}(\Omega)\big),\\ \partial_t^{i}\nabla^{\alpha}F \in
L^{\infty}\big(0, T; H^{k - 2i - |\alpha|}(\Omega)\big).
\end{array}\right.
\end{equation}
 for all $i, \alpha$ satisfying
$2i + |\alpha| \leq k$, $T$ being determined as that in step 1.
Namely, for all $i, \alpha$ satisfying $2i + |\alpha| \leq k$, we
have
\begin{equation}\label{i1}
\big(\|\partial_t^{i}\nabla^{k - 2i}v\|^{2} +
\|\partial_t^{i}\nabla^{k - 2i}E\|^{2}\big) +
\int_{0}^T\|\partial_t^{i}\nabla^{k + 1 - 2i}v\|^{2}\ dt < +
\infty.
\end{equation}
Here and in what follows the summations are performed over
repeated indices $i$ regardless of their position, as we had
assumed before. Our goal is to prove that the above results are
valid for all $j, \alpha$ satisfying $2j + |\alpha| \leq k + 1$,
which are equivalent to:
\begin{equation}\label{i2}
\big(\|\partial_t^{j}\nabla^{k + 1 - 2j}v\|^{2} +
\|\partial_t^{j}\nabla^{k + 1 - 2j}E\|^{2}\big) +
\int_{0}^T\|\partial_t^{j}\nabla^{k + 2 - 2j}v\|^{2}\ dt < +
\infty.
\end{equation}
where the summation over $j$ is from $0$ to $\frac{k}{2}$ if $k$
is an even number, and from $0$ to $\frac{k + 1}{2}$ if $k$ is an
odd number, respectively.
\par First, we assume that $k$ is an even
number and (\ref{i1}) is satisfied. By applying
$\partial_t^{j}\nabla^{k - 2j}$ to the second equation in
(\ref{e4}), we have
\begin{eqnarray}\label{i3}
&&\partial_t^{j}\nabla^{k - 2j}v_t +
\partial_t^{j}\nabla^{k - 2j}(v \cdot \nabla v) +
\partial_t^{j}\nabla^{k - 2j}\nabla p
\\\nonumber &&= \mu\partial_t^{j}\nabla^{k - 2j}\Delta v +
\partial_t^{j}\nabla^{k - 2j}\nabla \cdot (EE^T)
+ \partial_t^{j}\nabla^{k - 2j}\nabla \cdot E.
\end{eqnarray}

By taking the $L^{2}$ inner product of the equation (\ref{i3})
with $\partial_t^{j + 1}\nabla^{k - 2j}v$, $0 \leq j \leq
\frac{k}{2}$ and using integration by parts, we get
\begin{eqnarray}\label{i4}
&&\frac{\mu}{2}\frac{d}{dt}\|\partial_t^{j}\nabla^{k + 1 -
2j}v\|^{2} + \|\partial_t^{j + 1}\nabla^{k - 2j}v\|^{2}\\\nonumber
&&= -\ \big(\partial_t^{j}\nabla^{k - 2j}(v \cdot \nabla v),
\partial_t^{j + 1}\nabla^{k - 2j}v\big)  - \big(\partial_t^{j}\nabla^{k - 2j}\nabla p,
\partial_t^{j + 1}\nabla^{k - 2j}v\big)\\\nonumber &&\quad +\
\big(\partial_t^{j}\nabla^{k - 2j} \nabla \cdot (EE^T),
\partial_t^{j + 1}\nabla^{k - 2j}v\big)  +\ (\partial_t^{j}\nabla^{k - 2j}\nabla \cdot E,
\partial_t^{j + 1}\nabla^{k - 2j}v)\\\nonumber &&\leq
\|\partial_t^{j + 1}\nabla^{k -
2j}v\|\Big(\|\partial_t^{j}\nabla^{k - 2j}(v \cdot \nabla v)\|
\\\nonumber &&\quad +\ \|\partial_t^{j}\nabla^{k - 2j} \nabla \cdot
(EE^T)\| + \|\partial_t^{j}\nabla^{k - 2j} \nabla \cdot E\|\Big).
\end{eqnarray}

Applying Lemma \ref{lem31},  the induction assumption (\ref{i1})
yields
\begin{eqnarray*}\label{i5a}
&&\|\partial_t^{j}\nabla^{k - 2j}(v \cdot \nabla v)\|\\\nonumber
&&= \|\nabla^k(v \cdot \nabla v)\| + \|\partial_t^{\frac{k}{2}}(v
\cdot \nabla v)\|  + \sum_{0 < j < \frac{k}{2}}
\|\partial_t^{j}\nabla^{k - 2j}(v \cdot \nabla v)\|\\\nonumber
&&\leq \|v\|_{L^\infty}\|\nabla^{k + 1}v\| +
\|v\|_{L^\infty}\|\partial_t^{\frac{k}{2}}\nabla v\| + C\sum_{0 <
l \leq \frac{k}{2}}\|\nabla^lv\|_{L^4}\|\nabla^{k + 1 -
l}v\|_{L^4}
\\\nonumber &&\quad +\ C\sum_{0 < l \leq
\frac{k}{2}}\|\partial_t^lv\|_{L^4}\|\partial_t^{\frac{k}{2} -
l}\nabla v\|_{L^4} + \|v\|_{L^\infty}\|\partial_t^j\nabla^{k - 2j
+ 1}v\| \\\nonumber &&\quad +\ C\sum_{0 < j < \frac{k}{2}, (l, n)
\neq (j, k - 2j)}\|\partial_t^{j - l}\nabla^{k - 2j -
n}v\|_{L^{4}}\|\partial_t^{l}\nabla^{n}\nabla v\|_{L^{4}}.
\end{eqnarray*}
Further computation shows that:
\begin{eqnarray}\label{i5}
&&\|\partial_t^{j}\nabla^{k - 2j}(v \cdot \nabla v)\|\\\nonumber
&&\leq C\big(1 + \|\partial_t^{j}\nabla^{k + 1 - 2j}v\|\big) +
  C\sum_{0 < j < \frac{k}{2}, (l, n) \neq (j, k -
  2j)}\|\partial_t^{j - l}\nabla^{k - 2j - n}v\|^{1 - \theta(1)}
  \\\nonumber
&&\quad\times\ \|\partial_t^{l}\nabla^{n}\nabla v\|^{1 -
  \theta(1)}\cdot\ \|\partial_t^{j - l}\nabla^{k + 1 - 2j -
  n}v\|^{\theta(1)}\|\partial_t^{l}\nabla^{n + 1}\nabla v\|^{\theta(1)}\\\nonumber
&&\leq C\big(1 + \|\partial_t^{j}\nabla^{k + 1 - 2j}v\|\big).
\end{eqnarray}
In a similar way, we also have
\begin{equation}\label{i6}
\|\partial_t^{j}\nabla^{k - 2j} \nabla \cdot (EE^T)\| +
\|\partial_t^{j}\nabla^{k - 2j} \nabla \cdot E\| \leq C\big(1 +
\|\partial_t^{j}\nabla^{k + 1 - 2j}E\|\big).
\end{equation}
Putting these estimates (\ref{i5})-(\ref{i6}) into (\ref{i4}), we
obtain
\begin{eqnarray}\label{i7}
&&\frac{d}{dt}\|\partial_t^{j}\nabla^{k + 1 - 2j}v\|^{2} +
\|\partial_t^{j + 1}\nabla^{k - 2j}v\|^{2} \\\nonumber &&\leq
C\big(1 + \|\partial_t^{j}\nabla^{k + 1 - 2j}v\|^2 +
\|\partial_t^{j}\nabla^{k + 1 - 2j}E\|\big)^{2}.
\end{eqnarray}
\par By taking the $L^{2}$ inner product of the equation
(\ref{i3}) with $\partial_t^{j}\nabla^{k - 2j}\Delta v$, for  $0
\leq j \leq \frac{k}{2}$ and using integration by parts, we get
\begin{eqnarray}\nonumber
&&\mu\|\partial_t^{j}\nabla^{k + 2 - 2j}v\|^{2}\\\nonumber &&=
\big(\partial_t^{j + 1}\nabla^{k - 2j}v, \partial_t^{j }\nabla^{k
- 2j}\Delta v\big)+ \big(\partial_t^{j}\nabla^{k - 2j}(v \cdot
\nabla v),
\partial_t^{j }\nabla^{k - 2j}\Delta v\big)\\\nonumber &&\quad     -
\big(\partial_t^{j}\nabla^{k - 2j} \nabla \cdot (EE^T),
\partial_t^{j }\nabla^{k - 2j}\Delta v\big) -
\big(\partial_t^{j}\nabla^{k - 2j} \nabla \cdot E, \partial_t^{j
}\nabla^{k - 2j}\Delta v\big)\\\nonumber &&\leq
\|\partial_t^{j}\nabla^{k + 2 - 2j}v\|\Big(\|\partial_t^{j +
1}\nabla^{k - 2j}v\| + \|\partial_t^{j}\nabla^{k - 2j}(v \cdot
\nabla v)\|\\\nonumber &&\quad +\ \|\partial_t^{j}\nabla^{k - 2j}
\nabla \cdot (EE^T)\| + \|\partial_t^{j}\nabla^{k - 2j} \nabla
\cdot E\|\Big).
\end{eqnarray}
Noting (\ref{i5}), (\ref{i6}), we get
\begin{eqnarray}\label{i8}
&&\|\partial_t^{j}\nabla^{k + 2 - 2j}v\|^{2} \leq C\big(1 +
\|\partial_t^{j + 1}\nabla^{k - 2j}v\|^{2}\\\nonumber &&\quad +
\|\partial_t^{j}\nabla^{k + 1 - 2j}v\|^{2}  +
\|\partial_t^{j}\nabla^{k + 1 - 2j}E\|^{2}\big).
\end{eqnarray}
\par Applying
$\partial_t^{j}\nabla^{k + 1 - 2j}$ to the third equation in
(\ref{e4}) gives
\begin{equation}\label{i9}
\partial_t^{j}\nabla^{k + 1 - 2j}E_t
+ \partial_t^{j}\nabla^{k + 1 - 2j}(v \cdot \nabla E) =
\partial_t^{j}\nabla^{k + 1 - 2j}(\nabla v E)
+ \partial_t^{j}\nabla^{k + 1 - 2j}\nabla v.
\end{equation}
Now, we take the $L^{2}$ inner product of (\ref{i9}) with
$\partial_t^{j}\nabla^{k + 1 - 2j}E,\ 0\ \le j\ \leq\
\frac{k}{2}$, and use integration by parts:
\begin{eqnarray}\label{i10}
&&\frac{1}{2}\frac{d}{dt}\|\partial_t^{j}\nabla^{k + 1 -
2j}E\|^{2}\\\nonumber&&= \big(\partial_t^{j}\nabla^{k + 1 -
2j}(\nabla v E),
\partial_t^{j }\nabla^{k + 1 - 2j}E\big) + \big(\partial_t^{j}\nabla^{k + 2 -
2j}v , \partial_t^{j }\nabla^{k + 1 - 2j}E\big)\\\nonumber&&\quad
-\ \big(\partial_t^{j}\nabla^{k + 1 - 2j}(v \cdot \nabla E),
\partial_t^{j }\nabla^{k + 1 - 2j}E\big) \\\nonumber
&&\leq \|\partial_t^{j }\nabla^{k + 1 -
2j}E\|\big(\|\partial_t^{j}\nabla^{k + 1 - 2j}(\nabla v E)\| +
\|\partial_t^{j}\nabla^{k + 2 - 2j}v\|\big)\\\nonumber &&\quad +\
\|\partial_t^{j }\nabla^{k + 1 - 2j}E\|\|\partial_t^{j}\nabla^{k +
1 - 2j}(v \cdot \nabla E) - v \cdot \nabla \partial_t^{j}\nabla^{k
+ 1 - 2j}E\|.
\end{eqnarray}
By a similar process as in (\ref{i5}), we can have
\begin{eqnarray}\label{i11}
&&\|\partial_t^{j}\nabla^{k + 1 - 2j}(\nabla v E)\|\\\nonumber
&&\leq \|\partial^{\frac{k}{2}}_t\nabla(\nabla vE)\| + \|\nabla^{k
+ 1}(\nabla vE)\| + \|\partial_t^{j}\nabla^{k + 2 -
2j}v\|\|E\|_{L^\infty}\\\nonumber &&\quad + \ \|\nabla
v\|_{L^\infty}\|\partial_t^{j}\nabla^{k + 1 - 2j}E\|\\\nonumber
&&\quad +\ \sum_{0 < j < \frac{k}{2},}\sum_{(l, n) \neq (0, 0),
(j, k + 1 - 2j)} \|\partial_t^{j - l}\nabla^{k + 2 - 2j -
n}v\|_{L^4}\|\partial_t^l\nabla^nE)\|_{L^4}\\\nonumber &&\leq C(1
+ \|\nabla\Delta v\|)\big(\|\partial_t^{j}\nabla^{k + 2 - 2j}v\| +
\|\partial_t^{j}\nabla^{k + 1 - 2j}E\|\big).
\end{eqnarray}
On the other hand, we can estimate the last line of (\ref{i10}) as
follows
\begin{eqnarray}\label{i12}
&&\|\partial_t^{j}\nabla^{k + 1 - 2j}(v \cdot \nabla E) - v \cdot
\nabla \partial_t^{j}\nabla^{k + 1 - 2j}E\| \\\nonumber &&\leq
\|\nabla v\|_{L^\infty}\|\nabla\partial_t^{j}\nabla^{k -
2j}E\|\\\nonumber &&\quad +\ C\sum_{(l, n) \neq}\sum_{(j, k - 2j),
(j, k + 1 - 2j)}\|\partial_t^{j - l}\nabla^{k + 1 - 2j -
n}v\|_{L^4}\|\partial_t^l\nabla^{n + 1}E)\|_{L^4}\\\nonumber &&
\leq C(1 + \|\nabla\Delta v\|)\big(\|\partial_t^{j}\nabla^{k + 2 -
2j}v\| + \|\partial_t^{j}\nabla^{k + 1 - 2j}E\|\big).
\end{eqnarray}

Combining (\ref{i8}) with (\ref{i10})-(\ref{i12}), we have
\begin{eqnarray}\label{i13}
&&\frac{d}{dt}\|\partial_t^{j}\nabla^{k + 1 -
2j}E\|^{2}\\\nonumber &&\leq C(1 + \|\nabla\Delta
v\|)\big(\|\partial_t^{j}\nabla^{k + 1 - 2j}v\|^{2} +
\|\partial_t^{j}\nabla^{k + 1 - 2j}E\|^{2}\big)\\\nonumber &&
 +\ \frac{1}{2}\|\partial_t^{j + 1}\nabla^{k - 2j}v\|^{2}.
\end{eqnarray}
Combining this formula (\ref{i13}) with (\ref{i7}), we can
conclude that
$$\frac{d}{dt}\big(\|\partial_t^{j}\nabla^{k + 1 - 2j}v\|^{2} +
\|\partial_t^{j}\nabla^{k + 1 - 2j}E\|^{2}\big) + \|\partial_t^{j
+ 1}\nabla^{k - 2j}v\|^{2}$$$$ \leq C\big(1 + \|\nabla\Delta
v\|)(\|\partial_t^{j}\nabla^{k + 1 - 2j}v\|^{2} +
\|\partial_t^{j}\nabla^{k + 1 - 2j}E\|^{2}\big) + C.\ \ $$ Noting
that $\int_{0}^T|\nabla v|_{L^{\infty}}\ dt < \infty$, we can
apply Gronwall's inequality to get
\begin{eqnarray}\label{i14}
&&\big(\|\partial_t^{j}\nabla^{k + 1 - 2j}v\|^{2} +
\|\partial_t^{j}\nabla^{k + 1 - 2j}E\|^{2}\big) \\\nonumber
&&\quad + \int_{0}^t\|\partial_t^{j + 1}\nabla^{k - 2j}v\|^{2}\ ds
\leq M,\ \ \ 0 \leq t \leq T
\end{eqnarray}
where $M$ depending only on $\|v_{0}\|_{H^{k + 1}}$ and
$\|E_{0}\|_{H^{k + 1}}$. Moreover, by (\ref{i8}), we have
\begin{equation}\label{i15}
\int_{0}^T\|\partial_t^{j}\nabla^{k + 2 - 2j}v\|^{2}\ ds \leq M.
\end{equation}
Together, (\ref{i14}) and (\ref{i15}) imply (\ref{i2}) when $k$ is
an even number.

We now assume that $k$ is an odd number and $k \geq 3$. Applying
$\partial_t^{j}\nabla^{k + 1 - 2j}$ to the second and third
equation of (\ref{e4}), we have
\begin{equation}\label{j1}
      \left\{\begin{array}{l@{\quad\ \quad}l}
\partial_t^{j}\nabla^{k + 1 - 2j}v_t
+ \partial_t^{j}\nabla^{k + 1 - 2j}(v \cdot \nabla v) +
\partial_t^{j}\nabla^{k + 1 - 2j}\nabla p \\
\quad = \mu\partial_t^{j}\nabla^{k + 1 - 2j}\Delta v  +\
\partial_t^{j}\nabla^{k + 1 - 2j}\nabla \cdot (EE^T)
+ \partial_t^{j}\nabla^{k + 1 - 2j}\nabla \cdot E,\\
\partial_t^{j}\nabla^{k + 1 - 2j}E_t +
\partial_t^{j}\nabla^{k + 1 - 2j}(v \cdot \nabla E) =
\partial_t^{j}\nabla^{k + 1 - 2j}(\nabla v E)\\
\quad + \partial_t^{j}\nabla^{k + 1 - 2j}\nabla v.
\end{array}\right.
\end{equation}

Now we take the $L^{2}$ inner product of the first equation in the
system (\ref{j1}) with $\partial_t^{j}\nabla^{k + 1 - 2j}v,$ where
$ 0 \leq j \leq \frac{k + 1}{2}$,  integration by parts yields the
following
\begin{eqnarray*}
&&\frac{1}{2}\frac{d}{dt}\|\partial_t^{j}\nabla^{k + 1 -
2j}v\|^{2}
+ \mu\|\partial_t^{j}\nabla^{k + 2 - 2j}v\|^{2}\\
&&= -\big(\partial_t^{j}\nabla^{k + 1 - 2j}(v \cdot \nabla v),
\partial_t^{j}\nabla^{k + 1 - 2j}v\big) \\&&\quad +\ \big(\partial_t^{j}\nabla^{k +
1 - 2j} \nabla \cdot (EE^T), \partial_t^{j}\nabla^{k + 1 -
2j}v\big)
\\&&\quad +\ \big(\partial_t^{j}\nabla^{k + 1 - 2j}
\nabla \cdot E, \partial_t^{j}\nabla^{k + 1 - 2j}v\big)\\
&&\leq \|\partial_t^{j}\nabla^{k + 2 -
2j}v\|\Big(\|\partial_t^{j}\nabla^{k + 1 - 2j}(v \otimes v)\| +
\|\partial_t^{j}\nabla^{k + 1 - 2j}(EE^T)\| \\&&\quad +\
\|\partial_t^{j}\nabla^{k + 1 - 2j}E\|\Big)\\
&&\leq C\|\partial_t^{j}\nabla^{k + 2 - 2j}v\|\Big(1 +
\|\partial_t^{j}\nabla^{k + 1 - 2j}v\| + \|\partial_t^{j}\nabla^{k
+ 1 - 2j}E\|\Big)
\end{eqnarray*}
where we used Lemma \ref{lem31} and the induction assumption. In
summary, we have
\begin{eqnarray}\label{j2}
&&\frac{d}{dt}\|\partial_t^{j}\nabla^{k + 1 - 2j}v\|^{2} +
\mu\|\partial_t^{j}\nabla^{k + 2 - 2j}v\|^{2}\\\nonumber
 &&\leq C\big(1
+ \|\partial_t^{j}\nabla^{k + 1 - 2j}v\|^{2} +
\|\partial_t^{j}\nabla^{k + 1 - 2j}E\|^{2}\big).
\end{eqnarray}

Similarly, we will take the $L^{2}$ inner product of second
equation of (\ref{j1}) with $\partial_t^{j}\nabla^{k + 1 - 2j}E, 0
\leq j \leq \frac{k + 1}{2}$ and use integration by parts. The
similar derivations as in (\ref{i12}) will give us
\begin{eqnarray*}
&&\frac{1}{2}\frac{d}{dt}\|\partial_t^{j}\nabla^{k + 1 -
2j}E\|^{2}\\
&&= -\big(\partial_t^{j}\nabla^{k + 1 - 2j}(v \cdot \nabla E),
\partial_t^{j}\nabla^{k + 1 - 2j}E\big) + \big(\partial_t^{j}\nabla^{k + 1 -
2j}
(\nabla v E), \partial_t^{j}\nabla^{k + 1 - 2j}E\big)\\
&&\quad +\ \big(\partial_t^{j}\nabla^{k + 1 - 2j}
\nabla v, \partial_t^{j}\nabla^{k + 1 - 2j}E\big)\\
&&\leq \|\partial_t^{j}\nabla^{k + 1 -
2j}E\|\Big(\|\partial_t^{j}\nabla^{k + 1 - 2j}(v \cdot \nabla E) -
v \cdot
\partial_t^{j}\nabla^{k + 1 - 2j}\nabla E\|\\
&&\quad + \|\partial_t^{j}\nabla^{k + 1 - 2j}(\nabla v E)\| +
\|\partial_t^{j}\nabla^{k + 2 - 2j}v\|\Big)\\
&&\leq C\|\partial_t^{j}\nabla^{k + 1 - 2j}E\|\Big(1 + (1 +
|\nabla
E|_{L^{\infty}})\|\partial_t^{j}\nabla^{k + 1 - 2j}v\|\\
&&\quad +\ (1 + |\nabla v|_{L^{\infty}})\|\partial_t^{j}\nabla^{k
+ 1 - 2j}E\| + \|\partial_t^{j}\nabla^{k + 2 - 2j}v\|\Big).
\end{eqnarray*}
Employing the induction assumption, we have
\begin{eqnarray}\label{j3}
&&\frac{1}{2}\frac{d}{dt}\|\partial_t^{j}\nabla^{k + 1 -
2j}E\|^{2}\\\nonumber &&\leq C\big(1 + \|\partial_t^{j}\nabla^{k +
1 - 2j}v\|^{2} + \|\partial_t^{j}\nabla^{k + 1 - 2j}E\|^{2}\big) +
\frac{\mu}{4}\|\partial_t^{j}\nabla^{k + 2 - 2j}v\|^{2}.
\end{eqnarray}
\par Combining (\ref{j2}) with (\ref{j3}), we obtain:
\begin{eqnarray*}
&&\frac{d}{dt}\big(\|\partial_t^{j}\nabla^{k + 1 - 2j}v\|^{2} +
\|\partial_t^{j}\nabla^{k + 1 - 2j}E\|^{2}\big) +
\mu\|\partial_t^{j}\nabla^{k + 2 - 2j}v\|^{2}\\ &&\leq C\Big(1 +
\|\partial_t^{j}\nabla^{k + 1 - 2j}v\|^{2} +
\|\partial_t^{j}\nabla^{k + 1 - 2j}E\|^{2}\Big).
\end{eqnarray*}

Again, we use the Gronwall's inequality to deduce that
\begin{eqnarray}\label{j4}
&&\big(\|\partial_t^{j}\nabla^{k + 1 - 2j}v\|^{2} +
\|\partial_t^{j}\nabla^{k + 1 - 2j}E\|^{2}\big)\\\nonumber &&\quad
+\ \int_{0}^t\|\partial_t^{j}\nabla^{k + 2 - 2j}v\|^{2}\ ds \leq
M,\ \ \ 0 \leq t \leq T.
\end{eqnarray}
This conclude the proof of (\ref{i2}) when $k$ is odd.

Putting all these results (\ref{i14}), (\ref{i15}) and (\ref{j4})
together, we have  proved  (\ref{f1}) and completed the proof of
Theorem \ref{thm31}.

\end{proof}

\section{Global Existence}

We now turn our attention to  the proof of  the global existence
of classical solution for system (\ref{e4}). A weak dissipation on
the deformation $F$ is found by introducing an auxiliary function
$w$ below. The way of defining such a function reveals the
intrinsic dissipative nature of the system.

To avoid complications at the boundary, we only present the
periodic case $\Omega = \mathbf{T}^n$ and the whole space case
$\Omega = R^n$. In fact, the case of smooth bounded domain can
also be treated at a more lengthy, but no more difficult procedure
than the proofs presented here.

 Unlike those previous results in viscoelastic literature
\cite{Lions,Renardy,Renardy1}, the main difficulty lies in  the
apparent partial dissipation structure of the system (\ref{e4}).

On the other hand, it also lacks the property of scaling
invariance. The presence of viscosity on $v$ gives a big obstacle
to utilize the combination of Klainerman's generalized energy
estimates and weighted $L^2$ estimates
\cite{Klainerman2,Klainerman4,Sideris1,Sideris2,Sideris3}.

The main contribution of our work is to reveal the fact that the
incompressibility of system (\ref{e4}) will provide us enough
information for the proof of the near-equilibrium global existence
of classical solutions.

In the 3-D cases, the term $\nabla \times E$ is in fact a high
order term!  We recover the results obtained in \cite{Lin3}, where
we avoided using this fact by the introduction of the auxiliary
vector $\phi$ and then $\det\nabla\phi = 1$ is enough to prove the
near-equilibrium global existence of classical solutions in 2-D
case.

We start the proof by applying $\Delta$ to the transport equation
in (\ref{e4}) and then taking the $L^{2}$ inner product of the
resulting equation with $\Delta E$,
\begin{eqnarray}\label{k2}
&&\frac{1}{2}\frac{d}{dt}\|\Delta E\|^{2} - (\Delta\nabla v ,
  \Delta E)\\\nonumber
&&= -\ \big(\Delta(v \cdot \nabla E), \Delta E\big) +
  \big(\Delta(\nabla v E), \Delta E\big)\\\nonumber
&&\leq C\|\Delta E\|\Big(\|\Delta E\||\nabla v|_{L^{\infty}} +
  \|\Delta v\|_{L^{4}}\|\nabla E\|_{L^{4}} +
  \|\nabla\Delta v\||E|_{L^{\infty}}\Big)\\\nonumber
&&\leq C\|\Delta E\|^{2}\big(\|\nabla v\| + \|\nabla\Delta
  v\|\big) + C\|\Delta E\|\|\nabla\Delta
  v\|\|E\|_{H^2}\\\nonumber
&&\quad +\ C\|\Delta E\|\big(\|\nabla v\| + \|\nabla\Delta
  v\|\big)\big(\|\Delta E\| + \|E\|\big)\\\nonumber
&&\leq C\|E\|_{H^{2}}\|\Delta E\|\big(\|\nabla v\|+ \|\nabla\Delta
  v\|\big)\\\nonumber
&&\leq \ C\|E\|_{H^{2}}\Big(\|\Delta E\|^2 + \|\nabla v\|^2 +
  \|\nabla\Delta v\|^2\Big).
\end{eqnarray}

Next we apply $\Delta$ to the momentum equation in (\ref{e4}) and
then take the $L^{2}$ inner of the resulting equation with $\Delta
v$ to deduce that
\begin{eqnarray}\label{k3}
&&\frac{1}{2}\frac{d}{dt}\|\Delta v\|^{2} + \mu\|\nabla\Delta
  v\|^{2}\\\nonumber
&&= -\big(\Delta(v \cdot \nabla v), \Delta v\big) +
  \big(\Delta\nabla \cdot(EE^T), \Delta v\big) + \big(\Delta\nabla
  \cdot E, \Delta v\big)\\\nonumber
&&\leq C\|\Delta v\|\|\Delta v\||\nabla v|_{L^{\infty}} +
  C\|\Delta E\||E|_{L^{\infty}}\|\nabla\Delta v\| - (\Delta E,
  \nabla\Delta v)\\\nonumber
&&\leq C\big(\|v\|_{H^{2}} + \|E\|_{H^{2}}\big)\Big(\|\nabla
  v\|^{2} + \|\nabla\Delta v\|^{2} + \|\Delta
  E\|^2\Big)\\\nonumber
&&\quad  - (\Delta E, \nabla\Delta v),
\end{eqnarray}
where in the first inequality, we used Proposition \ref{prop32}.

Combining (\ref{k2}) with (\ref{k3}), we arrive at
\begin{eqnarray}\label{k4}
&&\frac{1}{2}\frac{d}{dt}\Big(\|\Delta v\|^{2} + \|\Delta
  E\|^{2}\Big) + \mu\|\nabla\Delta v\|^{2}\\\nonumber
&&\leq C\big(\|v\|_{H^{2}} + \|E\|_{H^{2}}\big)\Big(\|\nabla
  v\|^{2} + \|\nabla\Delta v\|^{2} + \|\Delta E\|^2\Big).
\end{eqnarray}
\par  In order to extract the dissipative nature of the system,
we want to combine the linear terms on the right hand side of the
momentum equation in (\ref{e4}).  We introduce the auxiliary
variable $w$ as follows:
\begin{equation}\label{k5}
w = \Delta v + \frac{1}{\mu}\nabla\cdot E.
\end{equation}
The system (\ref{e4}) will give the reformed equation:
\begin{eqnarray}
w_t &+& \Delta(v \cdot \nabla v) +
\frac{1}{\mu}\nabla\cdot(v\cdot\nabla E) + \Delta\nabla
p\\\nonumber
 &=& \mu\Delta w + \Delta\nabla\cdot (EE^T) +
\frac{1}{\mu}\nabla\cdot(\nabla v E) + \frac{1}{\mu}\Delta v.
\end{eqnarray}

By taking the $L^{2}$ inner product of the resulting equation with
$w$, we find
\begin{eqnarray}\label{k6}
&&\frac{1}{2}\frac{d}{dt}\|w\|^{2} + \mu\|\nabla
  w\|^{2}\\\nonumber
&&= - \big(\Delta(v \cdot \nabla v) + \frac{1}{\mu}
  \nabla\cdot(v\cdot\nabla E), w\big)\\\nonumber
&&\quad -\ (\Delta\nabla p, w) + \frac{1}{\mu}(\Delta v, w)
  \\\nonumber
&&\quad +\ \frac{1}{\mu}\big(\nabla\cdot(\nabla v E), w)\big) +
  \big(\Delta\nabla \cdot (EE^T), w\big).
\end{eqnarray}

Now let us estimate the right side of (\ref{k6}) term by term.
First of all, the first term can be estimated as
\begin{eqnarray}\label{k7}
&&\big|- \big(\Delta(v \cdot \nabla v) +
  \frac{1}{\mu}\nabla\cdot(v\cdot\nabla E), w\big)\big|\\\nonumber
&&\leq \big|(v\cdot\nabla w, w)\big| + \big|(\Delta(v\cdot\nabla
v) -
  v\cdot\nabla\Delta v, w)\big|\\\nonumber
&&\quad +\ \frac{1}{\mu}\big|\big(\nabla\cdot(v\cdot\nabla E) -
  v\cdot\nabla \nabla\cdot E, w\big)\big|\\\nonumber
&&\leq \big|(\Delta(v\otimes v) -
  v\otimes\Delta v, \nabla w)\big|\\\nonumber
&&\quad +\ \frac{1}{\mu}\big|\big(\nabla\cdot(v\otimes E) -
  v\otimes \nabla\cdot E, \nabla w\big)\big|\\\nonumber
&&\leq C\Big(\|\nabla v\|_{L^{4}}^2 + \|\Delta v\|\|v\|_{L^\infty}
  + \frac{1}{\mu}\|\nabla v\|\|E\|_{L^\infty}\Big)\|\nabla
  w\|\\\nonumber
&&\leq C\big(1 + \frac{1}{\mu}\big)\big(\| v\|_{H^{2}} + \|
  E\|_{H^{2}}\big)\Big(\|\nabla w\|^{2} + \|\nabla v\|^{2} + \|\nabla\Delta
  v\|^{2}\Big).
\end{eqnarray}
Next, we estimate the last term on the right hand side of
(\ref{k6}) as follows:
\begin{eqnarray}\label{k8}
&&\Big|\frac{1}{\mu}\big(\nabla\cdot(\nabla v E), w)\big) +
  \big(\Delta\nabla \cdot (EE^T), w\big)\Big|\\\nonumber
&&\leq C\|\nabla w\|\|E\|_{L^{\infty}}\big(\|\Delta E\| +
  \frac{1}{\mu}\|\nabla v\|\big)\\\nonumber
&&\leq C\big(1 + \frac{1}{\mu}\big)\| E\|_{H^{2}}\Big(\|\nabla
w\|^{2} + \|\nabla
  v\|^{2} + \|\Delta E\|^2\Big).
\end{eqnarray}
Here we used Proposition \ref{prop32}.

It is rather easy to get
\begin{equation}\label{k9}
\Big|\frac{1}{\mu}(\Delta v, w)\Big| \leq \frac{\mu}{4}\|\nabla
w\|^2 + \frac{C}{\mu^3}\|\nabla v\|^2.
\end{equation}

At last, let us estimate the term $(\nabla\Delta p, w)$. Noting
that $\nabla \cdot v = 0$ and (\ref{b4}), by applying the
divergence operator to the momentum equation of (\ref{e4}), we get
$$\Delta p =  \nabla_jE_{ik}\nabla_iE_{jk} -  \nabla_iv_j\nabla_j
v_i.$$ By Lemma \ref{lem31}, we have
\begin{equation}\label{80}
\|\nabla E\|_{L^{4}}^{2} \leq \left\{\begin{array}{l@{\quad\
  \quad}l}
  \|E\|^{\frac{1}{2}}\|\Delta E\|^{\frac{3}{2}}
  \leq \|E\|_{H^{2}}\|\Delta E\|,
  & \mbox{in}\quad R^2,\\\nonumber
  \|E\|^{\frac{1}{4}}\|\Delta E\|^{\frac{7}{4}}
  \leq \|E\|_{H^{2}}\|\Delta E\| & \mbox{in}\quad R^3,
\end{array}\right.
\end{equation}
This gives us the following estimates:
\begin{eqnarray}\label{k10}
&&\big|(\Delta\nabla p, w)\big| \leq \|\nabla w\|\Big(\|\nabla
E\|_{L^{4}}^{2} + \|\nabla v\|_{L^{4}}^{2}\Big)\\\nonumber &&\leq
C\big(\|E\|_{H^{2}} + \|v\|_{H^{2}}\big)\Big(\|\nabla w\|^{2} +
\|\nabla\Delta v\|^{2} + \|\nabla v\|^{2} + \|\Delta E\|^2\Big).
\end{eqnarray}

Combining all the above tedious but standard estimates
(\ref{k6})-(\ref{k10}) together, we arrive at the following
important energy inequality for the auxiliary variable $w$:
\begin{eqnarray}\label{k11}
&&\frac{d}{dt}\|w\|^{2} + \mu\|\nabla w\|^{2}\\\nonumber &&\leq
\frac{C}{\mu^3}\|\nabla v\|^2 + C\big(1 +
\frac{1}{\mu}\big)\big(\| v\|_{H^{2}} + \|
E\|_{H^{2}}\big)\\\nonumber &&\quad \times\Big(\|\Delta E\|^{2} +
\|\nabla v\|^{2} + \|\nabla\Delta v\|^{2} + \|\nabla w\|^2\Big).
\end{eqnarray}

The key here is to estimate the tern $\Delta E$.  Recall the Hodge
decomposition
$$\Delta E = \nabla\nabla\cdot E - \nabla\times\nabla\times E.$$
Here is the place that we will use (\ref{b5}) and (\ref{k5}) to
obtain the following estimate:
\begin{eqnarray}
&&\|\Delta E\|^2 = \|\nabla\nabla\cdot E\|^2 +
\|\nabla\times\nabla\times E\|^2\\\nonumber &&\leq
2\mu^2\Big(\|\nabla w\|^2 + \|\nabla\Delta v\|^2\Big) +
\|\nabla\times\nabla\times E\|^2\\\nonumber &&\leq
2\mu^2\Big(\|\nabla w\|^2 + \|\nabla\Delta v\|^2\Big) +
C\|E\|_{H^2}^2\|\Delta E\|^2,
\end{eqnarray}
which gives us the bound
\begin{equation}\label{k12}
\|\Delta E\|^2 \leq C\mu^2\Big(\|\nabla w\|^2 + \|\nabla\Delta
v\|^2\Big)
\end{equation}
provided $\|E\|_{H^2} \leq \frac{1}{\sqrt{2C}}$.

\par With the above result, we are ready to employ the same method
as that in \cite{Lin3} to prove the global existence results.
Combining (\ref{k4}), (\ref{k11}) with (\ref{k12}), we finally
arrive at
\begin{eqnarray*}
&&\frac{d}{dt}\Big(\|w\|^{2} + \|\Delta E\|^{2} + \|\Delta
v\|^{2}\Big) +
\mu\Big(\|\nabla w\|^{2} + \|\nabla\Delta v\|^{2}\Big)\\
&&\leq C\big(\mu^2 + \frac{1}{\mu}\big)\big(\| v\|_{H^{2}} + \|
E\|_{H^{2}}\big)\Big(\|\nabla w\|^{2} + \|\nabla v\|^{2} +
\|\nabla\Delta v\|^{2}\Big) + \frac{C}{\mu^3}\|\nabla
v\|^2.\end{eqnarray*} Thus, if the initial data is sufficiently
small, we can find some $T^\star
> 0$, such that
\begin{equation}\label{k13}
\| v\|_{H^{2}} + \| E\|_{H^{2}} \leq \frac{\mu^2}{2C(\mu^3 + 1)}
\end{equation}
for all $0 \leq t \leq T^\star$. Moreover, in this case,
\begin{eqnarray}
&&\Big(\|w\|^{2} + \|\Delta E\|^{2} + \|\Delta v\|^{2}\Big)(t) +
\mu\int_0^t\Big(\|\nabla w\|^{2} + \|\nabla\Delta
v\|^{2}\Big)d\tau\\\nonumber &&\leq C\big(\mu^2 +
\frac{1}{\mu^2}\big)\Big(\|v_0\|_{H^{2}}^2 + \|
E_0\|_{H^{2}}^2\Big) + \frac{C}{\mu^3}\int_0^\infty\|\nabla
v\|^2dt
\end{eqnarray}
holds for all $0 \leq t \leq T^\star$. Noting the original basic
energy law (\ref{g1}), we have
\begin{equation}\label{k14}
\Big(\|E\|_{H^2}^{2} + \|v\|_{H^2}^{2}\Big)(t) +
\mu\int_0^t\|\nabla v\|^{2}_{H^2}d\tau \leq C\big(\mu^2 +
\frac{1}{\mu^4}\big)\Big(\|v_0\|_{H^{2}}^2 + \|
E_0\|_{H^{2}}^2\Big)
\end{equation}
holds for all $0 \leq t \leq T^\star$. (\ref{k13}) (\ref{k14})
imply that if
\begin{equation}
\|v_{0}\|_{H^{2}}^2 + \|E_{0}\|_{H^{2}}^2 < \frac{\mu^8}{8C^3(1 +
\mu^6)(1 + \mu^3)^2}, \end{equation} then (\ref{k13}) is still
true with $\leq$ being replaced by $<$ for all $0 \leq t \leq
T^\star$, which implies that (\ref{k13}) is true for all the
latter time with the uniform constant $C$ independent of $t$ and
$\mu$. Moreover, from (\ref{k14}), we have
$$\|E\|_{H^2}^{2} + \|v\|_{H^2}^{2} +
\mu\int_0^\infty\|\nabla v\|^{2}_{H^2}dt \leq
\frac{\mu^2}{2C(\mu^3 + 1)}.$$ This together with the local
theorem \ref{thm31} gives the following global existence of
near-equilibrium classical solutions for system
(\ref{e4}).\\

Finally, we state the theorem in the lightly more general cases.
The proof is exactly the same as the case of (\ref{e4}).

\begin{theorem}\label{thm41}
Consider the viscoelastic model (\ref{a1}) with the initial data
(\ref{a2}) in the whole space $R^n$ or n-dimensional torus $T^n$,
for $n = 2, 3$. Suppose that the initial data satisfies the
incompressible constraint (\ref{a3}), and the strain energy
function satisfies the strong Legendre-Hadamard ellipticity
condition (\ref{e2}) and the reference configuration stress free
condition (\ref{w}). Then there exists a unique global classical
solution for system (\ref{a1}) which satisfies
$$\|E\|_{H^2}^{2} + \|v\|_{H^2}^{2} +
\mu\int_0^\infty\|\nabla v\|^{2}_{H^2}dt \leq
\frac{\mu^2}{2C(\mu^3 + 1)}$$ if the initial data $v_0,\ E_0 \in
H^k(\Omega)$ and satisfies the condition:
$$\|v_{0}\|_{H^{2}}^2 + \|E_{0}\|_{H^{2}}^2
< \frac{\mu^8 }{M(1 + \mu^{12})},$$ where $k$ is an integer and $k
\geq 2$, $M> 8C^3$ is a large enough constant.
\end{theorem}

\section{Incompressible Limits}
In numerical simulations and physical applications, one often
views the incompressible system as an approximation of the
compressible equations when the Mach number is small enough. Thus,
it is of interest to see whether the solution to the
incompressible system can be obtained as the incompressible limit
of  the corresponding compressible system. Moreover,
incompressible limit is also very important in the mathematical
understanding of different hydrodynamical systems and has been
extensively studied \cite{Klainerman3,Lei1,Lei3,Sideris3}.

The corresponding compressible viscoelastic system takes of the
following form:
\begin{equation}\label{l1}
      \left\{\begin{array}{l@{\quad\ \quad}l}
\partial_{t}\rho + v \cdot \nabla\rho + \rho\nabla\cdot v  = 0,\\
\partial_{t}v + v \cdot \nabla v +
\lambda^{2}\frac{p^{\prime}(\rho)}{\rho}\nabla\rho
 = \frac{\mu}{\rho}\big(\Delta v + \nabla(\nabla\cdot v)\big) +
\frac{1}{\rho}\nabla\cdot(\rho FF^{T}),\\
\partial_{t}F + v \cdot\nabla F = \nabla uF.
\end{array}\right.
\end{equation}
where $p(\rho)$ is a given equation of state independent of the
large parameter $\lambda$ with $p^{\prime}(\rho) > 0 $ for $\rho >
0$, and $\lambda$ the reciprocal of the Mach number $M$. For
simplicity, we only concern the Cauchy problem of system
(\ref{l1}). The initial data takes
\begin{equation}\label{l2}
\rho^{\lambda}(0, x) = 1 + \widetilde{\rho}_{0}^{\lambda}(x), \ \
v^{\lambda}(0, x) = v_{0}(x) + \widetilde{v}_{0}^{\lambda}(x), \ \
F^{\lambda}(0, x) = F_{0}(x) + \widetilde{F}_{0}^{\lambda}(x).
\end{equation}
where $\rho^\lambda(0, x)$, $F^\lambda(0, x)$ satisfy
$$\rho^\lambda(0, x)\det F^\lambda(0, x) = 1,$$
$v_{0}(x)$, $F_{0}(x)$ satisfy the incompressible constraints
(\ref{a3}) and $\widetilde{\rho}_{0}^{\lambda}(x)$,
$\widetilde{v}_{0}^{\lambda}(x)$, $\widetilde{F}_{0}^{\lambda}(x)$
are assumed to satisfy
\begin{equation}\label{l3}
\|\widetilde{\rho}_{0}^{\lambda}(x)\|_{s} \leq
\delta_{0}/\lambda^{2},\ \ \|\widetilde{v}_{0}^{\lambda}(x)\|_{s +
1} \leq \delta_{0}/\lambda,\ \
\|\widetilde{F}_{0}^{\lambda}(x)\|_{s} \leq \delta_{0}/\lambda.
\end{equation}
Here $\delta_{0}$ is a small positive constant and $s$ is an
integer with $s \geq 4$.

For the above system, we can state the following theorem:

\begin{theorem}\label{thm51}
The global classical solution for system (\ref{a1})-(\ref{a2}) can
be viewed as the incompressible limit of system
(\ref{l1})-(\ref{l2}) if (\ref{a3}), (\ref{e2}), (\ref{w}) and
(\ref{l3}) hold and the incompressible initial data satisfies
$$\|v_0\|_{H^s}^2 + \|E_0\|_{H^s}^2 \leq \varepsilon_0$$
for a sufficiently small constant $\varepsilon_0$.
\end{theorem}

The proof of Theorem \ref{thm51} relies on  the following Lemma
\ref{lem51}, namely, the uniform energy estimates with respect to
the parameter $\lambda$, which was proved in \cite{Lei3} in 2-D
case. The methods to prove the lemma, as well as the theorem, are
very similar in the 3-D cases here. We will not repeat the process
and
want to refer to \cite{Lei3} for details.\\

\begin{lemma}\label{lem51}
Consider the local solutions of the compressible viscoelastic
model (\ref{l1})-(\ref{l2}) under the constraints (\ref{a3}),
(\ref{e2}), (\ref{w})  and (\ref{l3}). Then the solution
$(\rho^{\lambda},\ v^{\lambda},\ F^{\lambda})$ to system
(\ref{l1})-(\ref{l2}) satisfies the following estimates
\begin{equation}\label{ccc}
      \left\{\begin{array}{l@{\quad\ \quad}l}
E_s(V^{\lambda}(t)) + \mu\int_0^{t}\|\nabla v^{\lambda}\|_s^2 \
dt\leq C\epsilon_0,\\
E_{s - 1}(\partial_{t}V^{\lambda}(t)) + \mu\int_0^{t}\|\nabla
\partial_t v^{\lambda}\|_{s-1}^2 \
dt\leq C\exp {Ct}.
\end{array}\right.
\end{equation}
 for any $t\in [0,
T^\lambda]$ and a universal constant C independent of $\lambda$ if
the initial data satisfies
$$\|v_{0}\|_{H^s}^2 + \|E_{0}\|_{H^s}^2 < \varepsilon_0.$$
Here $\varepsilon_0$ is a small enough constant and the energy
$E_s(V^{\lambda}(t))$ is defined as
$$E_s(V^{\lambda}(t)) = \|\lambda(\rho^\lambda - 1)\|_{H^s}^2
+ \|v^\lambda\|_{H^s}^2 + \|E^\lambda\|_{H^s}^2.$$ Moreover
$T_\lambda \rightarrow \infty$, as $\lambda\rightarrow + \infty$.
\end{lemma}

\section*{acknowledgement} Z. Lei was partially supported by the
National Science Foundation of China under grant 10225102 and
Foundation for Candidates of Excellent Doctoral Dissertation of
China. C. Liu was partially supported by National Science
Foundation grants NSF-DMS 0405850 and NSF-DMS 0509094. Y. Zhou was
partially supported by the National Science Foundation of China
under grant 10225102 and a 973 project of the National Sciential
Foundation of China. The authors also want to thank Professors
Weinan E, Fanghua Lin and Noel Walkington for many helpful
discussions.

%

\begin{thebibliography}{}

\bibitem{Agemi} \textsc{R. Agemi}: Global existence of nonlinear elastic
waves. \textit{Invent. Math.} \textbf{142(2)}, 225--250 (2000)

\bibitem{Alinhac1} \textsc{S. Alinhac}: \textit{Blowup for nonlinear
hyperbolic equations}. Birkh$\ddot{a}$user Boston, Boston, 1995

\bibitem{Alinhac2} \textsc{S. Alinhac}: The null condition for
quasilinear wave equations in two space dimensions. I.
\textit{Invent. Math.} \textbf{145(3)}, 597--618 (2001)

\bibitem{Alinhac3} \textsc{S. Alinhac}: The null condition for
quasilinear wave equations in two space dimensions. II.
\textit{Amer. J. Math.} \textbf{123(6)}, 1071--1101 (2001)

\bibitem{Bird} \textsc{R. Byron Bird, Charles F. Curtiss, Robert C.
Armstrong and Ole Hassager}: \textit{Dynamics of Polymeric
Liquids, Vol. 2, Kinetic Theory, 2 edition}. Wiley-Interscience,
New York, 1987

\bibitem{ChZh}  \textsc{Y. Chen and P. Zhang}: The Global Existence of
Small Solutions to the Incompressible Viscoelastic Fluid System in
General Space Dimensions.  preprint.

\bibitem{Christodoulou}  \textsc{D. Christodoulou}: Global existence
of nonlinear hyperbolic equations for small data. \textit{Comm.
Pure. Appl. Math.} \textbf{39}, 267--286 (1986)

\bibitem{Dafermos}  \textsc{C. Dafermos}: \textit{Hyperbolic Conservation
Laws in Continuum physics}. Springer, Heidelberg, 2000

\bibitem{Davison}  \textsc{P. A. Davidson}: \textit{An Introduction to
Magnetohydrodynamics}. Cambridge Texts in Applied Mathematics,
Cambridge University Press, 2001

\bibitem{deGennes} \textsc{P. de Gennes}: \textit{Physics of Liquid
Crystals}. Oxford University Press, London , 1976


\bibitem{Gurtin}  \textsc{M. E. Gurtin}: \textit{An introduction to
continuum mechanics}. Academic Press, New York, 1981

\bibitem{Joseph} \textsc{D. Joseph}: Instability of the rest state
of fluids of arbitrary grade greater than one. \textit{Arch.
Rational Mech. Anal.} \textbf{75(3)}, 251--256 (1980/81)

\bibitem{Kawashima}  \textsc{S. Kawashima and Y. Shibata}: Global
existence and exponential  stability of small solutions to
nonlinear viscoelasticity. \textit{Commum. Math. Phys.}
\textbf{148}, 189--208 (1992)

\bibitem{Klainerman1} \textsc{S. Klainerman}: Uniform decay estimates
and the Lorentz invariance of the classical wave equation.
\textit{Comm. Pure. Appl. Math.} \textbf{38}, 321--332 (1985)

\bibitem{Klainerman2} \textsc{S. Klainerman}:  The null condition and
global existence to nonlinear wave equations. \textit{Lect. in
Appl. Math.} \textbf{23}, 293--326 (1986)

\bibitem{Klainerman3} \textsc{S. Klainerman and A. Majda}:  Singular
limits of quasilinear hyperbolic system with large parameters and
the incompressible limit of compressible fluids. \textit{Comm.
Pure Appl. Math.} \textbf{34}, 481--524 (1981)

\bibitem{Klainerman4} \textsc{S. Klainerman and T. C. Sideris}: On almost
global existence for nonrelativistic wave equations in 3D.
\textit{Comm. Pure Appl. Math.} \textbf{49}, 307--322 (1996)

\bibitem{Ladyzhenskaya} \textsc{O. A. Ladyzhenskaya and G. A. Seregin}:
On the regularity of solutions of two-dimensional equations of the
dynamics of fluids with nonlinear viscosity. \textit{Zapiski
Nauchn. Semin. POMI} \textbf{259},145--166 (1999)

\bibitem{Larson}  \textsc{R. G. Larson}: \textit{The structure and
rheology of complex fluids}. Oxford University Press, New York,
1995

\bibitem{Lei1} \textsc{Z. Lei}: Global existence of classical
solutions for some Oldroyd-B model via the incompressible limit.
\textit{Chin. Ann. Math. Ser.B}, \textbf{27(5)}, 565--580 (2006)

\bibitem{Lei2} \textsc{Z. Lei, C. Liu and Y. Zhou}: Global existence for
small strain viscoelasticity. Preprint.


\bibitem{Lei3} \textsc{Z. Lei and Y. Zhou}: Global existence of
classical solutions for 2D Oldroyd model via the incompressible
limit. \textit{SIAM J. Math. Anal.} \textbf{37(3)}, 797--814
(2005)

\bibitem{Lin1} \textsc{F. H. Lin and C. Liu}:  Nonparabolic dissipative
systems modelling the flow of liquid crystals. \textit{Comm. Pure
Appl. Math.} \textbf{48(5)}, 501--537 (1995)

\bibitem{Lin2}\textsc{F. H. Lin and C. Liu}: Existence of solutions
for Erichsen-Leslie system. \textit{Arch. Ration. Mech. Anal.}
\textbf{154(2)}, 135--156 (2000)

\bibitem{Lin3} \textsc{F. H. Lin,  C. Liu and P. Zhang}: On
hydrodynamics of viscoelastic fluids. \textit{Comm. Pure Appl.
Math.} \textbf{58(11)}, 1437--1471 (2005)

\bibitem{Lions} \textsc{J. L. Lions}: \textit{On some questions in boundary
value problems of mathematical physics, in Contemporary
Development in Continuum Mechanics and PDE's}. North-Holland,
Amsterdam, 1978

\bibitem{Liu} \textsc{C. Liu and N. J. Walkington}: An Eulerian
description of fluids containing visco-hyperelastic particles.
\textit{Arch. Rat. Mech Ana.} \textbf{159}, 229--252 (2001)

\bibitem{MNR}
\textsc{J. M$\acute{a}$lek, J. Ne$\check{c}$as and K. R.
Rajagopal}: Global analysis of solutions of the flows of fluids
with pressure-dependent viscosities. \textit{Arch. Ration. Mech.
Anal.} \textbf{165(3)}, 243--269 (2002)

\bibitem{Renardy}
\textsc{M. Renardy}: An existence theorem for model equations
resulting from kinetic theories of polymer solutions. \textit{SIAM
J. Math. Anal.} \textbf{22}, 313--327 (1991)

\bibitem{Renardy1}
\textsc{M. Renardy, W. J. Hrusa and J. A. Nohel}:
\textit{Mathematical Problems in Viscoelasticity}. Longman
Scientific and Technical; copublished in the US with John Wiley,
New York, 1987

\bibitem{Schowalter} \textsc{W. R. Schowalter}: \textit{Mechanics
of Non-Newtonian fluids}. Pergamon Press, New York, 1978

\bibitem{Sideris1}
\textsc{T. C. Sideris}:  Nonresonance and global existence of
prestressed nonlinear elastic waves. \textit{Ann. of Math.}
\textbf{151}, 849--874 (2000)

\bibitem{Sideris2} \textsc{T. C. Sideris and S. Y. Tu}: Global existence
for system of nonlinear wave equations in 3D with multiple speeds.
\textit{SIAM J. Math. Anal.} \textbf{33}, 477--488 (2001)

\bibitem{Sideris3} \textsc{T. C. Sideris and B. Thomases}: Global
existence for 3D incompressible isotropic elastodynamics via the
incompressible limit. \textit{Comm. Pure Appl. Math.} \textbf{57},
1--39 (2004)

\bibitem{Sideris4} \textsc{T. C. Sideris and B. Thomases}: Global Existence
for 3D Incompressible Isotropic Elastodynamics. \textit{Accepted
for publication on Comm. Pure Appl. Math.}

\bibitem{Slemrod} \textsc{M. Slemrod}: Constitutive relations for
Rivlin-Erichsen fluids bases on generalized rational
approximation. \textit{Arch. Ration. Mech. Anal.} \textbf{146(1)},
73--93 (1999)

\bibitem{Teman} \textsc{R. Teman}: \textit{Navier-Stokes equations}.
North Holland, Amsterdam, 1977

\bibitem{Liu2}
\textsc{P. Yue, J. Feng, C. Liu and J. Shen}: A diffuse-interface
method for simulating two-phase flows of complex fluids.
\textit{Journal of Fluid Mechanics} \textbf{515}, 293--317 (2004)


\end{thebibliography}
%

\end{document}